\documentclass[aap]{imsart}

\RequirePackage[OT1]{fontenc}
\RequirePackage{amsthm,amsmath}
\RequirePackage[authoryear, round]{natbib}
\RequirePackage[colorlinks,citecolor=blue,urlcolor=blue]{hyperref}

\usepackage{amssymb}
\usepackage[linesnumbered,ruled,vlined]{algorithm2e}
\usepackage{mathrsfs}
\usepackage{graphicx}
\usepackage{url}
\usepackage{upgreek}
\usepackage{mathrsfs}
\usepackage{mathtools}
\usepackage{amsfonts}
\usepackage{array}
\usepackage{bbm}
\usepackage{bm}
\usepackage{enumerate}
\usepackage{caption} 
\usepackage{subcaption}
\usepackage{longtable}
\usepackage{indentfirst}
\usepackage[utf8]{inputenc}
\usepackage{times}
\usepackage[mathscr]{eucal}
\usepackage[math]{easyeqn}
\usepackage{xifthen}
\usepackage{nomencl}
\makenomenclature

\startlocaldefs
\numberwithin{equation}{section}
\theoremstyle{plain}
\newtheorem{theorem}{Theorem}[section]
\newtheorem*{theorem*}{Theorem}
\newtheorem{lemma}{Lemma}[section]
\newtheorem{corollary}{Corollary}[section]
\newtheorem{definition}{Definition}[section]

\newtheorem{proposition}{Proposition}
\newtheorem{remark}{Remark}
\theoremstyle{plain}
\newtheorem{assumption}{Assumption}
\endlocaldefs

\renewcommand{\(}{$\,}
\renewcommand{\)}{\,$}

\renewcommand{\o}[1]{\bm{#1}}

\newcommand{\cc}[1]{\mathscr{#1}}
\def\ND{\cc{N}}

\def\eqdef{\stackrel{\operatorname{def}}{=}}

\renewcommand{\Gamma}{\varGamma}
\renewcommand{\Pi}{\varPi}
\renewcommand{\Sigma}{\varSigma}
\renewcommand{\Delta}{\varDelta}
\renewcommand{\Lambda}{\varLambda}
\renewcommand{\Psi}{\varPsi}
\renewcommand{\Phi}{\varPhi}
\renewcommand{\Theta}{\varTheta}
\renewcommand{\Omega}{\varOmega}
\renewcommand{\Xi}{\varXi}
\renewcommand{\Upsilon}{\varUpsilon}

\DeclareMathOperator*{\argmin}{argmin}

\newcommand*{\R}{\mathbb{R}}

\newcommand*{\C}{\mathbb{C}}

\newcommand*{\sym}{\mathrm{Sym}}
\renewcommand*{\H}{\mathbb{H}}

\newcommand*{\M}{\mathbb{M}}
\newcommand*{\A}{\mathbb{A}}

\def\T{\top}

\DeclareMathOperator{\diag}{diag}
\DeclareMathOperator{\tr}{tr}
\DeclareMathOperator{\range}{range}
\DeclareMathOperator{\rank}{rank}

\renewcommand*{\P}{\mathbb{P}}
\newcommand*{\E}{\mathbb{E}}
\DeclareMathOperator{\var}{Var}

\def\CONST{\mathtt{C}}

\newcommand*{\lmax}{\lambda_{max}}
\newcommand*{\lmin}{\lambda_{min}}
\newcommand*{\diff}{\,d}

\newcommand*{\dT}[2]{\o{d T}_{#1}^{#2}}
\newcommand*{\dt}[2]{\o{d t}_{#1}^{#2}}
\newcommand*{\Id}{\o{Id}}

\renewcommand{\bar}[1]{\overline{#1}}
\renewcommand{\tilde}[1]{\widetilde{#1}}

\begin{document}
\begin{frontmatter}
\title{Statistical inference for Bures-Wasserstein barycenters}
\runtitle{Inference for Bures-Wasserstein barycenters}

\begin{aug}
\author{\fnms{Alexey} \snm{Kroshnin}
\ead[label=e1]{kroshnin@math.univ-lyon1.fr}
}
\address{Institute for Information Transmission Problems RAS\\
National Research University Higher School of Economics\\
Universit{\'e} Claude Bernard Lyon 1\\
\printead{e1}}

\author{\fnms{Vladimir} \snm{Spokoiny}
\ead[label=e2]{spokoiny@wias-berlin.de}
}
\address{Weierstrass Institute for Applied Analysis and Stochastics\\
\printead{e2}}
\and
\author{\fnms{Alexandra} \snm{Suvorikova}
\ead[label=e3]{suvorikova@math.uni-potsdam.de}}
\address{University of Potsdam\\
\printead{e3}}

\runauthor{Kroshnin, Spokoiny, Suvorikova}


\end{aug}

\begin{abstract}
In this work we introduce the concept of Bures-Wasserstein barycenter \(Q_*\), that is essentially a Fr\'echet mean of some distribution \(\P\) supported on a subspace of positive semi-definite Hermitian operators \(\H_{+}(d)\).
We allow a barycenter to be constrained to some affine subspace of \(\H_{+}(d)\) and provide conditions ensuring its existence and uniqueness.
We also investigate convergence and concentration properties of an empirical counterpart of \(Q_*\) in both Frobenius norm and Bures-Wasserstein distance, and explain, how obtained results are connected to optimal transportation theory and can be applied to statistical inference in quantum mechanics.
\end{abstract}



\begin{keyword}
\kwd{Bures-Wasserstein barycenter}
\kwd{Wasserstein barycenter}
\kwd{Hermitian operators}
\kwd{Central limit theorem}
\kwd{Concentration}
\end{keyword}
\end{frontmatter}

\section{Introduction}
Space of finite-dimensional Hermitian operators \(\H(d)\) provides a powerful toolbox 
for data representation.
For instance, in quantum mechanics it is used for mathematical description
of physical properties of a quantum system, also known as observables. 
The reason is due to the fact that the measurements obtained in a physical experiment
should be associated to real-valued quantities. 
Hermitian operators posses real-valued spectrum and
satisfy the above requirement.
A subspace \(\sym(d) \subset \H(d) \) of real-valued 
symmetric matrices is also of great interest: 
points in \(\sym(d)\) are widely used for
description of systems in engineering applications, medical studies, neural sciences, evolutionary biology e.t.c.
Usually such data sets are considered to be randomly sampled from an unknown distribution \(\P\)~(\cite{goodnight1997bootstrap, calsbeek2009empirical, alvarez2011wide, del2017robust, gonzalez2017absolute}), and
statistical characteristics of \(\P\) such as, in particular, mean and variance,
appear to be of interest for further
planning of an experiment and analysis of obtained results fur further development of 
natural science models. 
The current study focuses on a space of positive semi-definite 
Hermitian matrices \(\H_{+}(d) \subset \H(d) \) 
and presents a possible approach to analysis and aggregation 
of relevant statistical information from data-sets, 
for which the linearity assumption might be violated. This makes classical Euclidean
definitions of mean and variance not sensitive enough to capture
effects of interest.   
This case appears extremely often in multiple contexts. 
As an example one can consider a data set that is
represented as probability measures which belong to the same scale-location family, e.g. some astronomic measurements~\cite{alvarez2018wide}, Example 4.6.
Non-linearity assumption requires the development of a novel toolbox suitable for further statistical analysis. 
In order to detect non-linear effects, 
we suggest to endow \(\H_{+}(d) \) with the Bures-Wasserstein
distance \(d_{BW}\) which is recently introduced in a seminal paper~\cite{bhatia2018bures}. It is defined as follows.
For any pair of positive matrices \(Q, S \in \H_{+}(d)\) it is written as:
\begin{equation}
\label{def:Bures}
d_{BW}^2(Q, S) = \tr Q + \tr S - 2\tr\left( Q^{1/2} S Q^{1/2}\right )^{1/2}.
\end{equation}
It is worth noting that being restricted to the space of symmetric positive definite matrices \(\sym_{++}(d) \), 
\(d_{BW}\) boils down to a classical 2-Wasserstein distance between measures
that belong to the same scale-location family (see e.g.~\cite{agueh2011barycenters}, Section 6
or~\cite{alvarez2011wide}).
A more detailed discussion on this particular choice of the distance is presented in Section~\ref{section:Bures}.

After choosing a proper distance, we are now ready to introduce a model of information aggregation the statistical properties of which are investigated in the current study. Let \(\P\) be a probability distribution 
supported on some set \(\mathcal{S} \subseteq \H_+(d)\). 
Further without loss of generality we assume, 
that \(\P\) assigns positive probability to the intersection of \(\mathcal{S}\)
with space of positive definite Hermitian matrices \(\H_{++}(d)\),
and that the spectrum of its elemnts \(S \in \mathcal{S} \) is on average bounded away from infinity:
\begin{assumption}
	\label{asm:main_0}
	\[
	\P\left(\mathcal{S}\cap \H_{++}(d)\right) > 0,
	\quad
	\E\tr S < +\infty.
	\]
\end{assumption}
Two statistically important characteristics of \(\P\) are Fr\'echet mean and Fr\'echet variance.
The former one can be regarded as
a typical representative of a data-set in hand, 
whereas the latter appears in analysis
on data variability, see e.g.~\cite{del2015statistical}.
We briefly recall both concepts below.
For an arbitrary point \(Q \in \H_+(d)\) Fr\'echet variance of \(\P\) is defined as
\[
\mathcal{V}(Q) \eqdef \int_{\mathcal{S}}d_{BW}^2(Q, S)d\P(S).
\]
Classical Fr\'echet mean of \(\P\) is a set of global minimizers of \(\mathcal{V}(Q)\):
\begin{equation}
\label{def:bar}
Q_* \in \argmin_{Q \in \H_{+}(d)}\mathcal{V}(Q).
\end{equation}
However, in many cases
we are interested in a minimizer,
that belongs to some affine sub-space \(\A\):
\begin{equation}
\label{def:bar_aff}
Q_* \in \argmin_{Q \in \H_{+}(d)\cap \A} \mathcal{V}(Q).
\end{equation}
For instance, such a necessity may arise while considering 
a random set of quantum density operators. For introduction to density operators theory one may look through~\cite{fano1957description}.
This example is considered in more details in Section~\ref{section:quantum}.
Note, that the setting~\eqref{def:bar_aff} covers
the setting~\eqref{def:bar}. So without loss of generality 
we further address only~\eqref{def:bar_aff}.
Obviously, the first crucial question concerns existence and
uniqueness of \(Q_*\). And positive answers on both 
issues, along with necessary conditions, are presented in Theorem~\ref{thm:uniqueness}.
This immediately allows us 
to define the {global} Fr\'echet variance of \(\P\) as \(\mathcal{V}_* \eqdef \mathcal{V}(Q_*)\).

Given an i.i.d. sample \(S_1,..., S_n\), 
\(S_i \overset{iid}{\sim} \P \), one constructs
an empirical analogue of \(\mathcal{V}(Q)\):
\[
\mathcal{V}_n(Q) \eqdef \frac{1}{n} \sum^n_{i = 1} d_{BW}^2(Q, S_i). 
\]
An empirical Fr\'echet mean and global empirical variance
also exist {and unique}:
\begin{equation}
\label{def:barn_aff}
Q_n = \argmin_{Q \in \H_{+}(d)\cap \A}\mathcal{V}_n(Q),
\quad
\mathcal{V}_n \eqdef \mathcal{V}_n(Q_n).
\end{equation}
These facts follow from Theorem~\ref{thm:uniqueness}.
This work studies convergence of the estimators \(Q_n\) and \(\mathcal{V}_n \) and investigate
concentration properties of both quantities.
The discussion 
of practical applicability of the obtained results is postponed to Section~\ref{section:examples}. There we explain their relation 
to optimal transportation theory and present a possible application to statistical analysis in quantum mechanics.

\subsection{Contribution of the present study}
\paragraph{Central limit theorem and 
	{concentration} of \(Q_n\)}
The first main result of this study concerns asymptotic normality of the approximation error of population Fr\'echet mean 
by its empirical counterpart:
\[
\sqrt{n} \left(Q_n - Q_*\right) \rightharpoonup \ND\left(0, \o{\Upxi}\right),
\]
where ``\(\rightharpoonup\)'' stands for weak convergence, and \(\o{\Upxi} \) is some covariance operator
acting on the linear subspace \(\M \subset \H(d)\) associated with affine subspace \(\A\).
From now on we use bold symbols e.g. \(\o{A}, \o{B}\), 
to denote operators, wheres classical ones i.e.\({A}, {B}\) stand for matrices or vectors.
This convergence result cannot be directly used for construction of asymptotic confidence sets
because it relies on the unknown covariance matrix \( \o{\Upxi} \).
However, Theorem~\ref{thm:CLT} ensures,
that this covariance matrix can be replaced by its empirical counterpart \( \o{\hat{\Upxi}_n} \):
\[
\sqrt{n} \, \o{\hat{\Upxi}_n}^{-1/2} \left(Q_n - Q_*\right) \rightharpoonup \ND\left(0, \Id\right),
\]
where \(\Id\) denotes an identity operator.
Along with asymptotic normality of \((Q_n - Q_*)\), we are interested in the limiting
distribution of \(\mathcal{L}\left(\sqrt{n} d_{BW}(Q_n, Q_*)\right)\). 
Corollary~\ref{thm:asymptotic} shows, that
\[
\mathcal{L}\left(\sqrt{n} d_{BW}(Q_n, Q_*)\right) \rightharpoonup \mathcal{L}\left(\norm{\xi}_{F}\right),
\]
where {\(\xi\) is some normally distributed} vector.
Data-driven asymptotic confidence sets for \(\sqrt{n}d_{BW}(Q_n, Q_*)\)
are obtained by replacement of \(\xi\) by
its empirical counterpart \(\xi_n\):
\[
{d_w\left(\mathcal{L}\left(\sqrt{n} d_{BW}(Q_n, Q_*)\right), 
	\mathcal{L}\left(\norm{\xi_n}_F\right) \right) \rightarrow 0},
\]
where \(d_{w} \) is a metric which induces weak convergence. 
Furthermore, we investigate concentration
properties of \(Q_n\) in both Frobenius norm
and \(d_{BW}\) metric. 
The following two bounds hold with h.p.:
\[
\norm{Q^{-1/2}_* Q_n Q^{-1/2}_* - I}_{F} \le \frac{\CONST (d + t)}{\sqrt{n}},
\quad
d_{BW}(Q_n, Q_*) \le \frac{\CONST (d + t)}{\sqrt{n}},
\]
where \(\CONST\) stands for some generic constant.
A more detailed discussion is presented in Theorem~\ref{theorem:concentration_Q}
and Corollary~\ref{corollary:concantration_in_dbw}
respectively.
It is worth noting that concentration results are obtained under 
assumption of sub-Gaussianity of \(S\) in the following sense:
\begin{assumption}[Sub-Gaussianity of \(\sqrt{\tr S}\)]
	\label{asm:subgauss}
	Let \(\sqrt{\tr S}\) be sub-Gaussian, i.e.\
	\[
	\P\left\{\sqrt{\tr S} \ge t\right\} \le B e^{- b t^2} \quad \text{for any}~ t \ge 0,
	\]
	with some constants \(B, b > 0\).
\end{assumption}

All above-mentioned results are closely connected 
to convergence and concentration of empirical \(2\)-Wasserstein barycenters.
For the sake of transparency
we postpone this discussion to Section~\ref{section:wasserstein}.

\paragraph{CLT and concentration for \(\mathcal{V}_n\)}
We also show asymptotic normality
of approximation error of 
\(\mathcal{V}_{*}\) by \(\mathcal{V}_n \)
and prove concentration of \(\mathcal{V}_n \):
\[
\sqrt{n} \left(\mathcal{V}_n -  \mathcal{V}_*\right) \rightharpoonup \ND\left(0, \var d_{BW}^2(Q_*, S)\right),
\]
\[
\left|\mathcal{V}_n- \mathcal{V}_* \right|
\le \max\left(\frac{c_1(B, b, t, d)}{\sqrt{n}}, \frac{c_2(B, b, t, d)}{{n}}\right),
\]
where the latter result holds with h.p.,
and \(c_1(B,b, t, d)\), \(c_2(B, b, t, d)\) are constants which
depend on sub-Gaussianity parameters
\(B, b\), dimension \(d\), and parameter \(t\).
See Theorem~\ref{theorem:CLT_V} and
Theorem~\ref{theorem:concentration_V} respectively.

The paper is organised as follows. 
Section~\ref{section:main} explains the obtained results in more details.
Section~\ref{section:examples} illustrates the connection to other scientific problems. 
Finally, Section~\ref{section:simulations} contains simulations and experiments on 
both artificial and real data-sets.

\section{Results}
\label{section:main}
This section presents obtained results in more details, 
and the first question we address is the particular choice of the distance.
\subsection{Bures-Wasserstein distance}
\label{section:Bures}
The original Bures metric appears in quantum mechanics
in relation to fidelity measure
between two quantum states and is used for measurement of quantum entanglement~\cite{marian2008bures, dajka2011distance}. Let \(\rho \), \(\sigma\) be
two quantum states. Mathematically speaking, this means that
\begin{equation}
    \label{def:density}
    \rho, \sigma \in \H_{+}(d),
    \quad
    \tr \rho = 1,
    \quad
    \tr \sigma = 1.
\end{equation}
Fidelity of these states is defined as \(\mathcal{F}(\rho, \sigma) = \left(\tr\sqrt{\rho^{1/2}\sigma \rho^{1/2}} \right)^2 \).
It quantifies ``closeness'' of \(\rho\) and \(\sigma\), see
~\cite{jozsa1994fidelity}. 
It is obvious, that in case of~\eqref{def:density} Bures-Wasserstein distance turns into
\begin{equation}
\label{def:fidelity}
d^2_{B}(\rho, \sigma) 
= 2\left(1 - \mathcal{F}^{1/2}(\rho, \sigma) \right).
\end{equation}

It is interesting to note, that the \(d_{BW}\) distance appears not only one of the central distances, used in quantum mechanics, but also an object of extensive investigation in transportation theory~\cite{takatsu2011wasserstein}.
Let \(\ND(0, Q) \) and
\(\ND(0, S) \) be two
centred Gaussian distributions.
Then 2-Wasserstein distance between them is written as
\[
d^2_{W_2}\left(\ND(0, Q), \ND(0, S) \right) = \tr Q + \tr S - 2\tr\left(Q^{1/2}S Q^{1/2} \right)^{1/2}.
\]
The case of Gaussian measures is naturally extended
to measures that belong to a same scale-location family~\cite{alvarez2018wide}.
In the last few years Wasserstein distance attracts a lot of attention of
data scientists and machine learning community, 
as it takes into account
geometrical similarities between objects, see e.g.
~\cite{gramfort2015fast, flamary2018wasserstein, montavon2016wasserstein}. Due to this fact \(d_{BW}\) satisfies the requirement of taking into account non-linearity of a data set under consideration.
For more information on optimal transportation theory we recommend ~\cite{villani_optimal_2009}.

Following~\cite{bhatia2018bures}, we continue to investigate properties of
\(d_{BW}(Q, S)\). The next lemma presents an alternative analytical expression for the distance.
\begin{lemma}
\label{lemma:def_T}
    Let \(Q, S \in \H_+(d)\) and \(Q \succ 0\). Then~\eqref{def:Bures} can be rewritten  as
    \[
    d_{BW}^2(Q, S) = \norm{\left(T_Q^S - I\right) Q^{1/2}}_F^2 = \tr \left(T_Q^S - I\right) Q \left(T_Q^S - I\right),
    \]
    where
    \begin{align}
        T_Q^S & \eqdef \argmin_{T : T Q T^* = S} \norm{(T - I) Q^{1/2}}_F \nonumber \\ 
        & = S^{1/2} \bigl(S^{1/2} Q S^{1/2} \bigr)^{-1/2} S^{1/2} 
         = Q^{-1/2} \bigl(Q^{1/2} S Q^{1/2}\bigr)^{1/2} Q^{-1/2}. \label{def:T_Q}
    \end{align}
   By $\bigl(S^{1/2} Q S^{1/2} \bigr)^{-1/2}$ we denote the pseudo-inverse matrix $\left(\bigl(S^{1/2} Q S^{1/2} \bigr)^{1/2}\right)^+$.
\end{lemma}

Note, that in optimal transportation theory
\(T_Q^S\) is referred to as an optimal push-forward (optimal transportation map) 
between two centred normal distributions
\(\mathcal{N}(0, Q)\) and \(\mathcal{N}(0, S)\). Following optimal transport notations it is denoted as
\(T^S_{Q}\#\mathcal{N}(0, Q) = \mathcal{N}(0, S)\).

For general notes on optimal transportation maps see~\cite{brenier1991polar}; for a particular case of scale-location and Gaussian families one may refer to~\cite{alvarez2018wide, takatsu2011wasserstein}.
Lemma~\ref{lemma:taylor_T} presents
differentiability of the optimal map \(T^S_Q\). It is one of the key-ingredients in the proof of main results of the present study.
Note, that in case of \(\A = \sym_{++}(d) \) differentiability of \(T^S_{Q} \) is obtained in~\cite{rippl2016limit}.
More technical details on properties of \(d_{BW}\) are presented in Section~\ref{sec:propeties_dbw_proof}. However, for better understanding of the
proofs of main results we highly recommend to at least look through Section~\ref{sec:proof_frechet} which is dedicated to investigation of
properties of \(T^S_{Q}\) and its differential \(\dT{Q}{S}\).

\subsection{Existence and uniqueness of \(Q_*\) and \(Q_n\)}
Along with investigation of properties of the distance in hand
and before moving to more general statistical questions, 
one should ask her- or himself, 
whether Fr\'echet mean \(Q_*\) exists and, if so, is it unique or not? 
Let \(\M \) be a linear subspace of \(\H(d)\) associated to \(\A\), i.e. the following representation holds: \(\A = Q_0 + \M\) for some \(Q_0 \in \H(d)\).
We further assume that \(\A\) has a non-empty intersection
with the space of positive definite operators:
\begin{assumption}
\label{asm:main}
    \(\H_{++}(d) \cap \A \neq \emptyset\), 
\end{assumption}
The next theorem ensures existence and uniqueness of the Fr\'echet mean~\eqref{def:bar_aff}.
\begin{theorem}[Existence and uniqueness of Fr\'echet mean \(Q_*\)]
\label{thm:uniqueness}
    Under Assumptions~\ref{asm:main_0}~and~\ref{asm:main} there exists a unique positive-definite barycenter \(Q_*\) of \(\P\): \(Q_* \succ 0\).
    Moreover, it is characterised as the unique solution of the equation
    \begin{equation}
    \label{def:uniqueness}
        \o{\Pi_\M} \E T_Q^S = \o{\Pi_\M} I, \quad Q \in \H_{++}(d),
    \end{equation}
    where \(\o{\Pi_\M}\) is the orthogonal projector onto 
    \(\M\).
\end{theorem}
Note, that this result generalises
the result for scale-location families in 2-Wasserstein space, presented in~\cite{alvarez2011wide}, Theorem 3.10 and originally obtained in a seminal work ~\cite{agueh2011barycenters}, Theorem 6.1.
Namely, if \(\mathbb{A} = \sym_{++}(d) \), then
\(Q_*\) exists, is unique, and is characterised as the unique solution 
of a fixed-point equation similar to~\eqref{def:uniqueness}
\[
 Q_* = \E \left(Q^{1/2}_* S Q^{1/2}_* \right)^{1/2} ~\Longleftrightarrow~ \E T^S_{Q_*} = I.
\]
Existence, uniqueness, and measurability of the estimator \(Q_n\) defined in~\eqref{def:barn_aff} is a direct corollary of the above theorem.
The proof of Theorem~\ref{thm:uniqueness} is presented in Section~\ref{section:CLT_proof}.

\subsection{Convergence of \(Q_n\) and \(d_{BW}(Q_n, Q_*)\)}
Armed with the knowledge about properties of \(d_{BW}\), \(Q_*\), and \(Q_n\), we are now equipped enough, so that to introduce the main results of the current study.
Theorem~\ref{thm:CLT} presents asymptotic
convergence of \(Q_n\) to \(Q_*\).
\begin{theorem}[Central limit theorem for the Fr\'echet mean]
\label{thm:CLT}
    Under Assumptions~\ref{asm:main_0} and \ref{asm:main}
    an approximation error rate of the Fr\'echet mean \(Q_*\)
    by its empirical counterpart \(Q_n\) is
    \[
    \sqrt{n} \left(Q_n - Q_*\right) \rightharpoonup \ND\left(0, \o{\Upxi}\right)~\tag{A}
    \]
    where \(\o{\Upxi}\) is a self-adjoint linear operator acting from \(\M\) to \(\M\) {defined in}~\eqref{def:Upxi}.
    Moreover, {if {\(\var\left(T^{S}_{Q_*}\right)\)} is non-degenerated, then}
    \[
    {\sqrt{n} \, \o{\hat{\Upxi}}^{-1/2}_n \left(Q_n - Q_*\right) \rightharpoonup \ND\left(0, (\Id)_\M\right)}, ~\tag{B}
    \]
    with \(\o{\hat{\Upxi}_n} \) is a data-driven empirical counterpart of \(\o{\Upxi} \) {defined in}~\eqref{def:Upxi_n}.
\end{theorem}
\begin{remark}
Here \((\o{A})_{\M}\) denotes a restriction of a quadratic form \(\o{A}\) to a subspace \(\M\):
        \[
        (\o{A})_{\M} \colon \M \to \M, \quad (\o{A})_{\M}(X) = \o{\Pi}_\M \o{A}(X), \; X \in \M.
        \]
\end{remark}
We intentionally postpone the explicit definitions of 
\(\o{\Upxi}\) and \(\o{\hat{\Upxi}_n}\), as they require 
an introduction of many technical details. This would make the description of main results less transparent.
The proof of the theorem relies on the Fr\'echet differentiablilty of
\(T_Q^S\)
in the vicinity of \(Q_*\):
\[
T^{S}_{Q_n} = S^{1/2} \left(S^{1/2} Q_n S^{1/2}\right)^{-1/2} S^{1/2},
\quad
T^{S}_{Q_n} \approx T_{Q_*}^S + \dT{Q_*}{S}(Q_n - Q_*),
\]
where \(\dT{Q_*}{S}\) is a differential of \(T_{Q}^S\) at point \(Q_*\).
Here we imply differentiability of \(T_{Q}^{S}\) by 
the lower argument \(Q\).

It is worth noting that the result (B) 
obtained in CLT enables construction 
of data-driven asymptotic confidence sets.
However,  
there might appear technical problems with inversion 
of the empirical covariance. For instance, numerical simulations show, that \(\o{\Upxi}\) can be degenerated if \(\P\) is supported on a set of diagonal matrices. This immediately raises a question
concerning the development of some other confidence set construction methodology based on re-sampling techniques which would simplify the process from computational point of view.
We consider this as a subject for further research.

As soon as the Bures-Wasserstein distance is the main tool for the analysis in \(\H_{+}(d)\), the convergence properties of \(d_{BW}(Q_n, Q_*)\)
are also of great interest. The next lemma is almost a straightforward corollary of Theorem~\ref{thm:CLT}.
\begin{corollary}[Asymptotic distribution of \(d_{BW}(Q_n, Q_{*})\)]
\label{thm:asymptotic}
    Under conditions of Theorem~\ref{thm:CLT} it holds
    \[
    \mathcal{L} \left(\sqrt{n} d_{BW}(Q_n, Q_*)\right) 
    \rightharpoonup 
    \mathcal{L}\left(\norm{Q^{1/2}_* \dT{Q_*}{Q_*}(Z) }_{F}\right),
    \]
    where \(Z \in \M \subset \H(d)\) is random matrix, s.t.\ 
    \(Z \sim \ND\left(0,  \o{\Upxi} \right)\).
    Moreover,
    \[
    d_w\Bigl(\mathcal{L}\left(\sqrt{n}d_{BW}(Q_n, Q_*)\right), 
    \mathcal{L}\left(\norm{Q^{1/2}_n \dT{Q_n}{Q_n}(Z_n)}_{F}\right)  \Bigr) \rightarrow 0,
    \]
    where \(Z_n \in \M\) and \(Z_n \sim \ND\left(0,  \o{\Upxi_n} \right)\).
\end{corollary}
To illustrate the result, we consider the case of
diagonal \(Q_*\). This setting allows us to write down 
the explicit form of the limiting distribution. If \(Q_* = \diag(q_1, ..., q_d)\), then
right-hand side of the above corollary for \(Z\)-case
is:
\[
\mathcal{L} \left(\sqrt{n} d_{BW}(Q_n, Q_*)\right)
\rightharpoonup 
\mathcal{L} \left(\sum_{i, j = 1}^d \frac{Z_{i j}^2}{2 (q_i + q_j)} \right),
\]
where \((Z_{i j})^{d}_{i, j = 1} = Z\).
All proofs are collected in Section~\ref{section:CLT_proof}.
Section~\ref{section:simulations} illustrates asymptotic behaviour of \(\mathcal{L}\left(\sqrt{n}\norm{Q_n - Q_*}_{F} \right) \) and \(\mathcal{L}\left(\sqrt{n} d_{BW}(Q_n, Q_*)\right)\) on both artificial and real data sets.

\subsection{Concentration of \(Q_n\)}
The next important issue is concentration 
properties of \(Q_n\) under
the assumption of sub-Gaussianity of \(\P\) (Assumption~\ref{asm:subgauss}).

\begin{theorem}[Concentration of \(Q_n'\)]
\label{theorem:concentration_Q}
    Let 
    \begin{equation}
        \label{def:Q'_n}
        Q'_n \eqdef Q^{-1/2}_* Q_n Q^{-1/2}_*.
    \end{equation}
    It holds under Assumptions~\ref{asm:subgauss} and~\ref{asm:main}, that
    \[
    \P\left\{\norm{Q_n' - I}_F 
    \ge \frac{c_Q}{\sqrt{n}} (d + t)\right\}
    \le 2 m e^{-n t_F} + e^{-t^2 / 2} + (1 - p)^n
    \]
    for any \(t \ge 0\) and \(n \ge n_0 \eqdef c_Q^2 (d + t)^2\),
    where 
    \begin{gather*}
        m \eqdef \dim(\M), \quad p \eqdef \P\bigl(\H_{++}(d)\bigr), \\
        c_Q \eqdef \frac{4 \norm{Q_*} \sigma_T}{\lmin(\o{F'})},
        \quad
        t_F \eqdef \CONST \min\left(\frac{\lmin(\o{F'})}{U \log^{1/2}\left(U / \sigma_F\right)}, \frac{\lmin^2(\o{F'})}{\sigma_F^2}\right),
    \end{gather*}
    operator \(\o{F'} \) is defined in~\eqref{def:F'}, \(\sigma_{T}\)  comes from Proposition~\ref{prop:concentr_T}, and \(\sigma_F\) and \(U\) are defined in Lemma~\ref{lemma:concentration_F}.
\end{theorem}

Concentration of \(d_{BW}(Q_n, Q_*) \) is a corollary of the above theorem.

\begin{corollary}[Concentration of \(Q_n\) in \(d_{BW}\) distance]
\label{corollary:concantration_in_dbw}
Under conditions of Theorem~\ref{theorem:concentration_Q}
the following result holds
\[
\P \left\{d_{BW}(Q_n, Q_*) \ge \tfrac{c_Q \norm{Q_*}^{1/2}}{\sqrt{n}} (d + t) \right\} \le 2 m e^{-n t_F} + e^{-t^2/2} + (1 - p)^n.
\] 
\end{corollary}
Proofs are collected in Section~\ref{sec:concentration_proof}.

\subsection{Central limit theorem and concentration for \(\mathcal{V}_n\)}
In this section we investigate properties of the
Fr\'echet variance \(\mathcal{V}_n\), defined
in~\eqref{def:barn_aff}.
The next theorem presents central limit theorem for empirical variance \(\mathcal{V}_n\).

\begin{theorem}[Central limit theorem for \(\mathcal{V}_n\)]
\label{theorem:CLT_V}
Let \(\P\) be s.t.\ \(\E (\tr S)^2 < \infty\) and \(\P\left(\H_{++}(d)\right) = 1\). Then
\[
\sqrt{n} \left(\mathcal{V}_n - \mathcal{V}_*\right) \rightharpoonup \ND\left(0, \var d_{BW}^2(Q_*, S)\right).
\]
\end{theorem}
The last important result of the current study describes concentration properties of \(\mathcal{V}_n\).

\begin{theorem}[Concentration of \(\mathcal{V}_n\)]
\label{theorem:concentration_V}
    Let Assumption~\ref{asm:subgauss} be fulfilled. Then under conditions of Theorem~\ref{theorem:CLT_V} it holds:
    \begin{align*}
    \P\left\{\left|\mathcal{V}_n - \mathcal{V}_* \right|
    \ge z(b, \nu, d, n ,t)\right\}
    \le 2 m e^{-n t_F} + 3 e^{-t^2 / 2} + (1 - p)^n
    \end{align*}
    with 
    \[
    z(b, \nu, d, n ,t) \eqdef \max\left(\tfrac{b t^2}{n}, \tfrac{\nu t}{\sqrt{n}}\right) + 3  \tfrac{c_Q^2 \norm{\o{F'}} }{n} (d + t)^2.
    \]
    There\((\nu, b)\) are parameters of 
    sub-exponential r.v.\ \(d^2_{BW}(Q_*, S)\).
\end{theorem}

Proofs of these two theorems are collected in Section~\ref{sec:clt_V_proof}.

\section{Connection to other problems}
\label{section:examples}
In this section we explain the connection of obtained results to some other problems. Section~\ref{section:wasserstein} investigates 
the relation between Bures-Wasserstein barycenter and 2-Wasserstein barycenter of some scale-location family.
Section~\ref{section:quantum} illustrates the idea of search of
a barycenter on some affine subspace \(\A \subset \H(d) \). 

\subsection{Connection to scale-location families of measures}
\label{section:wasserstein}
We first present the concept of a scale-location family
of absolutely continuous
measures supported on \(\R^d\).
\begin{definition}
\label{def:SL_family}
Let \(X \sim \mu\) be a random variable
that follows law \(\mu\):
\(\mu \in \mathcal{P}^{ac}_{2}(\R^d) \), where
\(\mathcal{P}^{ac}_{2}(\R^d) \) is a set of all
continuous measures with finite second moment. A set of all affine transformations of \(X\) is
\[
\mathcal{SL}(\mu) \eqdef \Bigl\{\mathcal{L}\bigl(P X + p \bigr)
:~
P \in \sym_+(d),~p \in \R^d
\Bigr\}.
\]
It is referred to as a scale-location family.
\end{definition}
Scale-location families attract lots of attention in modern data analysis 
and appear in many practical applications, 
as this concept is user-friendly in terms of theoretical analysis and, at the same time, 
possess very high modelling power. 
For example, it is widely used in
medical imaging~\cite{wassermann2010unsupervised}, modelling of molecular dynamic~\cite{gonzalez2017absolute},
clustering procedures~\cite{del2017robust},
climate modelling~\cite{fragen2017learning}, embedding of complex objects 
in low dimensional spaces~\cite{muzellec2018generalizing} and so on.

A possible metric that takes into account non-linearity of the underlying data-set is \(2\)-Wasserstein distance \(d_{W_2}\). 
Let \(\mu_1, \mu_2\) be elements of \( \mathcal{SL}(\mu)\) and let \(X \sim \mu_{1}\),
\(Y \sim \mu_2 \).
We denote their first and second moments as
\begin{equation}
\label{def:moments}
\E X = m_1, 
\quad
\E Y = m_2,
\quad 
\var(X) = S_1,
\quad
\var(Y) = S_2.
\end{equation}
It is a well-known fact,
that in case of scale-location families \(d_{W_2}\) depends 
only on the first and second moments of observed measures:
\[
d^2_{W_2}(\mu_1, \mu_2) = \norm{m_1 - m_2}^2 + d^2_{BW}(S_1, S_2).
\]
For more details on general class of optimal transportation
distances we recommend
excellent books~\cite{ambrosio2013user} or \cite{villani_optimal_2009}.

\paragraph{Distribution over scale-location family}
In many cases we are interested in scale-location families 
generated at random. Let \(\P\) be a probability measure 
supported on some \(\mathcal{SL}(\mu)\).
And let \(\left(\Omega, F, \P \right)\) be a generic 
probability space, s.t. for any \(\omega \in \Omega\)
there exists an image in \(\mu_{\omega} \eqdef \mathcal{L}\bigl(P_{\omega}X + p_{\omega} \bigr)\), where
\(P_{\omega} \in \sym_+(d)\) is a scaling parameter and 
\(p_{\omega} \in \R^d\) is a shift parameter. 
A randomly sampled measure \(\mu_{\omega}\) belongs to \(\mathcal{SL}(\mu) \) by construction,
and its first and second moments \((m_{\omega}, S_{\omega})\)
are written as
\[
m_{\omega} \eqdef P_{\omega}r + p_{\omega},
\quad
S_{\omega} \eqdef P_{\omega} Q P^{\T}_{\omega},
\]
where \((r, Q)\) denote
the first and the second moments of \(\mu\).

Fr\'echet variance of \(\P\) at any arbitrary point \(\mu\) is written as
\[
\mathcal{V}(\mu) \eqdef \int_{\text{supp}(\P)}d^2_{W_2}(\mu, \nu_{\omega})\P(d\omega).
\]
Given an  i.i.d. sample \(\nu_1,..., \nu_n\) from \(\P\), 
we define an empirical analogon of \(\mathcal{V}\):
\[
\mathcal{V}_n(\mu) \eqdef \frac{1}{n}\sum^n_{i=1}d^2_{W_2}(\mu, \nu_i).
\]
Then population and empirical barycenters \(\mu_*\) and \(\mu_n\)
are
\[
\mu_* = \argmin_{\mu \in \mathcal{P}_2(\R^d)} \mathcal{V}(\mu), 
\quad
\mu_n = \argmin_{\mu \in \mathcal{P}_2(\R^d)} \mathcal{V}_n(\mu).
\]
Note, that
\(\mu_*\) and \(\mu_n\) belong to \(\mathcal{SL}(\mu)\)
and are uniquely characterised by their first and second moments \((r_*, Q_*)\) and \((r_n, Q_n)\) respectively, 
see e.g. {Theorem 3.10}~\cite{alvarez2011wide}:
\begin{equation}
\label{def:r*}
 r_* = \int_{\text{supp}(\P)}m_{\omega}\P(d\omega), 
\quad
Q_* = \int_{\text{supp}(\P)}\left(Q^{1/2}_* S_{\omega}Q^{1/2}_*\right)^{1/2}\P(d\omega),   
\end{equation}

\begin{equation}
\label{def:rn}
r_n = \frac{1}{n}\sum^n_{i = 1}m_i,
\quad
Q_n = \frac{1}{n}\sum^n_{i = 1}\left(Q^{1/2}_n S_{i}Q^{1/2}_n\right)^{1/2}.    
\end{equation}

It is worth noting that the concept of Wasserstein barycenter originally presented in
a seminal work by ~\cite{agueh2011barycenters}
becomes a topic of extensive scientific interest in the last few years.
A work~\cite{bigot2012consistent} focuses on convergence of
parametric class of barycenters, while~\cite{bigot2017penalized} 
investigate asymptotic properties of regularised barycenters.
The most general results on limiting distribution of convergence of empirical barycenters 
are obtained in~\cite{ahidar2018rate}. 
This work provides rates of convergence for empirical barycentres of a Borel probability
measure on a metric space either under assumptions on
weak curvature constraint of the underlying space or for a case of a non-negatively curved space on which geodesics, emanating from a barycenter, can be extended.
Theorem~\ref{thm:CLT}
specifies the results, obtained in~\cite{ahidar2018rate}
for the case of scale-location families.
Corollary~\ref{thm:asymptotic} partially answers an (implicit) question,
raised by work\\ \mbox{\cite{gouic_existence_2015}}, 
concerned the rate of convergence
of \(d_{W_2}(\mu_n, \mu_*)\).
Namely, for the case of scale-location families
it is of order \(\frac{1}{\sqrt{n}}\).
However, the above mentioned work covers only~\eqref{def:bar} case.
The paper~\cite{kroshnin2017fr} obtains an analog of law of large numbers for the case of arbitrary cost functions for barycenters on some affine sub-space \(\A\)~\eqref{def:bar_aff}.
A result, similar in spirit to Theorem~\ref{theorem:concentration_V}
is obtained in~\cite{barrio2016central}. 
However, there authors consider only the space of probability measures supported on the real line (i.e. \(d = 1 \)) endowed
with 2-Wasserstein distance.
To the best of our knowledge, there are no results similar to concentration Theorem~\ref{theorem:concentration_Q} and Theorem~\ref{corollary:concantration_in_dbw} in case of 2-Wasserstein distance.

\subsection{Connection to quantum mechanics}
\label{section:quantum}
This section illustrates the idea of barycenter restricted to some affine sub-space \(\A\). We first briefly recall the concept of quantum densities. 
Quantum density operator is used in quantum mechanics 
as a possible way of description of statistical state of a quantum system.
It might be considered as an analogue to a phase-space density in classical statistical mechanics. The formalism was introduced by John von Neumann in 1927.  
In essence a density matrix \(\rho\) is a Hermitian positive semi-definite operator 
with the unit trace, \(\rho \in \H_{+}(d)\), \(\tr \rho = 1\).

Given a random ensemble of density
matrices, one is able to recovery the mean using averaging in classical Euclidean sense. 
However, Bures-Wasserstein barycenter suggests an alternative way
to define the ``most typical'' representant~\eqref{def:bar_aff}
in terms of fidelity measure~\eqref{def:fidelity}.
We consider a following statistical setting. 
Let \((\Omega, F, \P)\) be some mechanism 
which generates quantum 
states \(\rho_{\omega}\). 
Given an i.i.d sample \(\rho_1,..., \rho_n\) we write population and empirical variance of \(\P\) as
\[
\mathcal{V}(\sigma) = \int_{\text{supp}(\P)}d^2_{BW}(\sigma, \rho_{\omega})\P(d\omega),
\quad
\mathcal{V}_n(\sigma) = \frac{1}{n}\sum^n_{i = 1}d^2_{BW}(\sigma, \rho_i).
\]
Then
population and empirical barycenters in the class of
all \(d \times d\)-dimensional density operators are defined as
\[
\rho_* = \argmin_{\sigma:~ \tr \sigma = 1} \mathcal{V}(\sigma),
\quad
\rho_n = \argmin_{\sigma:~ \tr \sigma = 1} \mathcal{V}_n(\sigma).
\]
It can be easily shown, that ``taking global Fr\'echet barycenter'' or, in other words neglecting the condition \(\sigma:~ \tr \sigma = 1\), we end up with the global baryceneter, which is the solution of the fixed point equation which is already mentioned in Section~\ref{section:main}: \(
\rho = \int \left(\rho^{1/2}\rho_{\omega} \rho^{1/2} \right)^{1/2}\P(d\omega)\). This is a contraction mapping. Thus \(\tr \rho_* < 1\) and \(\rho_*\) is not a density operator. In other words
condition \(\tr \sigma = 1\) ensures, that \(\rho_*\) and \(\rho_n\) also belong to the class of density operators.
Taking into account the results obtained in Section~\ref{section:main}, \(\rho_n\) is a natural
consistent estimator of \(\rho_*\) with known
rate of convergence and deviation properties. 

\section{Experiments on simulated and real data sets}
\label{section:simulations}

\subsection{Simulated data}
In this section we consider a simulated
data set.
So as to generate a covariance matrix \(S = U^* \Lambda U\), \(S \in \sym_{++}(d) \)
we generate at random an orthogonal
matrix \(U\) and a diagonal matrix \(\Lambda = \text{diag}(\lambda_1,..., \lambda_d)\),
s.t. \(\lambda_i \sim \text{Unif}[18, 22] \).
The following images Fig.\ref{fig:fnorm} - Fig.\ref{fig:var}
illustrate convergence of \(\mathcal{L}\left(\sqrt{n}\| Q_n - Q_*\|_{F}\right)\), 
\(\mathcal{L}\left(\sqrt{n}d_{BW}(Q_n, Q_*)\right)\), and \(\mathcal{L}\left(\sqrt{n}(\mathcal{V}_* - \mathcal{V}_n)\right) \) 
presented in Theorem~\ref{thm:CLT}, Corollary~\ref{thm:asymptotic}, and Theorem~\ref{theorem:CLT_V} respectively. The following numerical experiments were performed using \({\textsf{R}} \).
The population barycenter \(Q_*\) was computed using a sample of \(20 000\) observed covariance matrices.
A solid line depicts the density of a limiting distribution, 
whereas dashed lines correspond to densities for
different sample sizes for Bures-Wasserstein barycenter
\(Q_n\) with \(n = \{3, 10, 100, 1000\}\).
Simulation were carried out for matrices of size \(d = 5\) and \(d = 10\).

\begin{figure}[h!]
	\centering
		\includegraphics[width=0.9\textwidth]{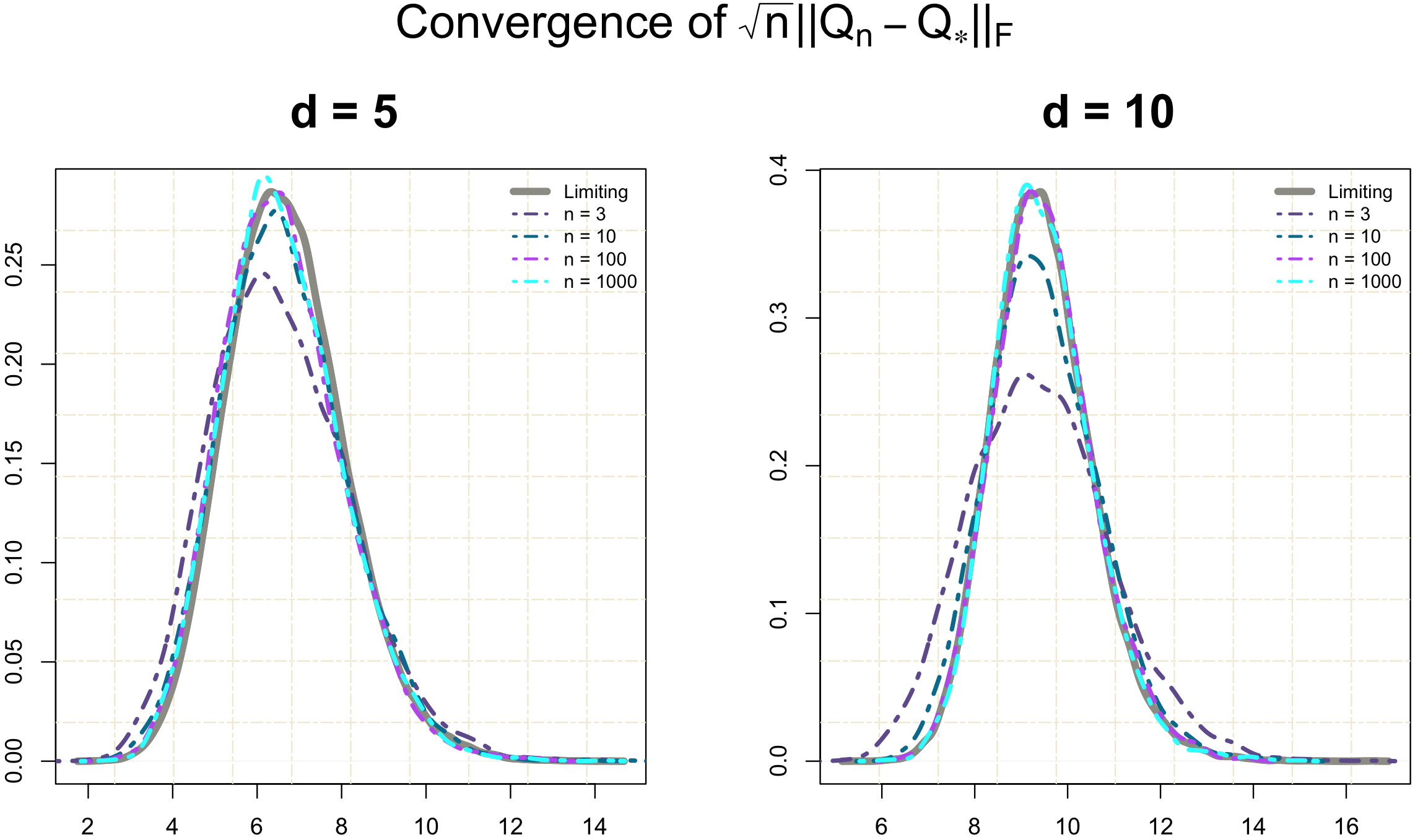}
	\caption{Densities of \(\mathcal{L}(\sqrt{n}\norm{Q_n - Q_*}_{F})\)}
	\label{fig:fnorm}
\end{figure}

\begin{figure}[h!]
	\centering
		\includegraphics[width=0.9\textwidth]{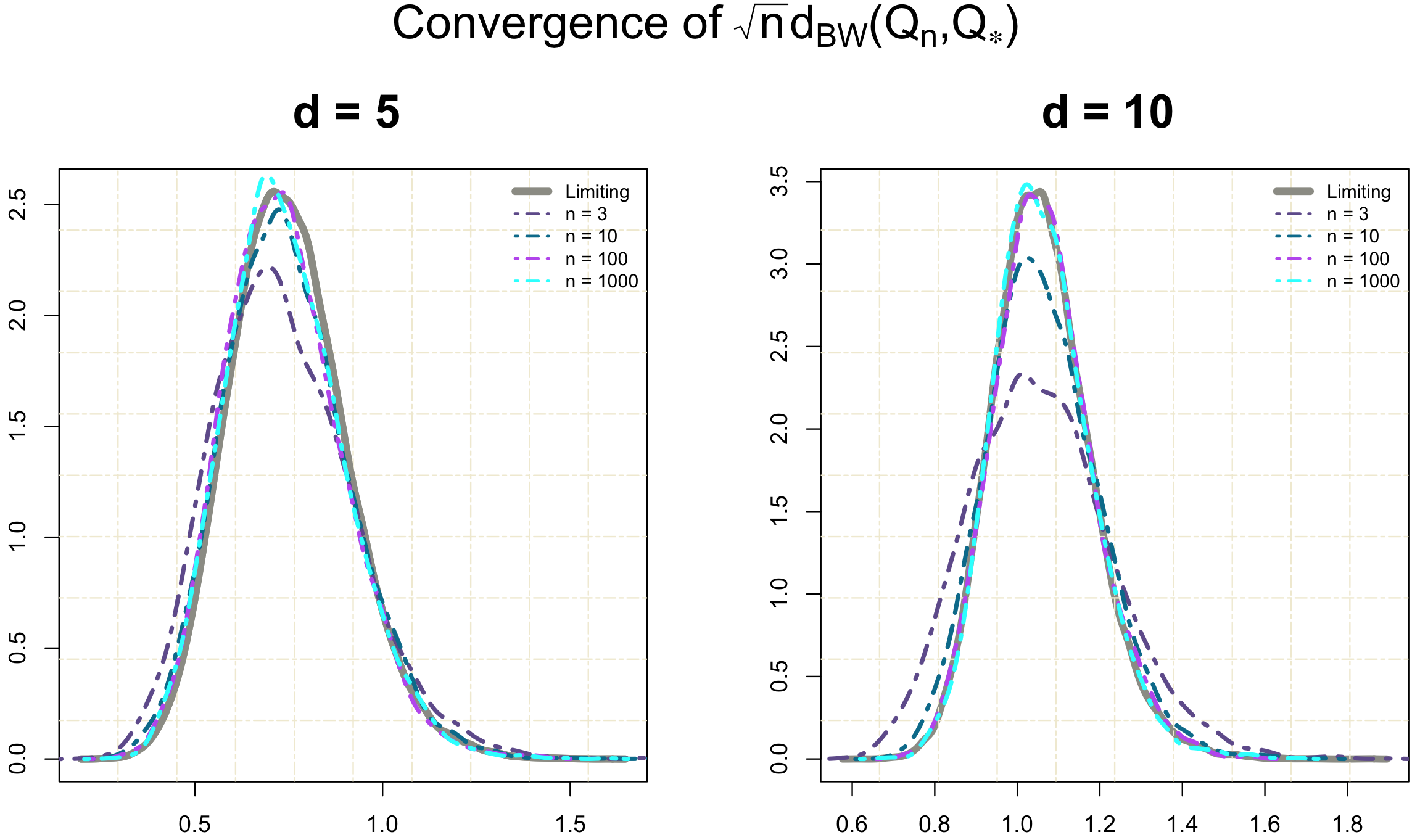}
	\caption{Densities of \(\mathcal{L}\left(\sqrt{n}d_{BW}(Q_n, Q_*)\right)\)}
	\label{fig:dbw}
\end{figure}

\begin{figure}[h!]
	\centering
		\includegraphics[width=0.9\textwidth]{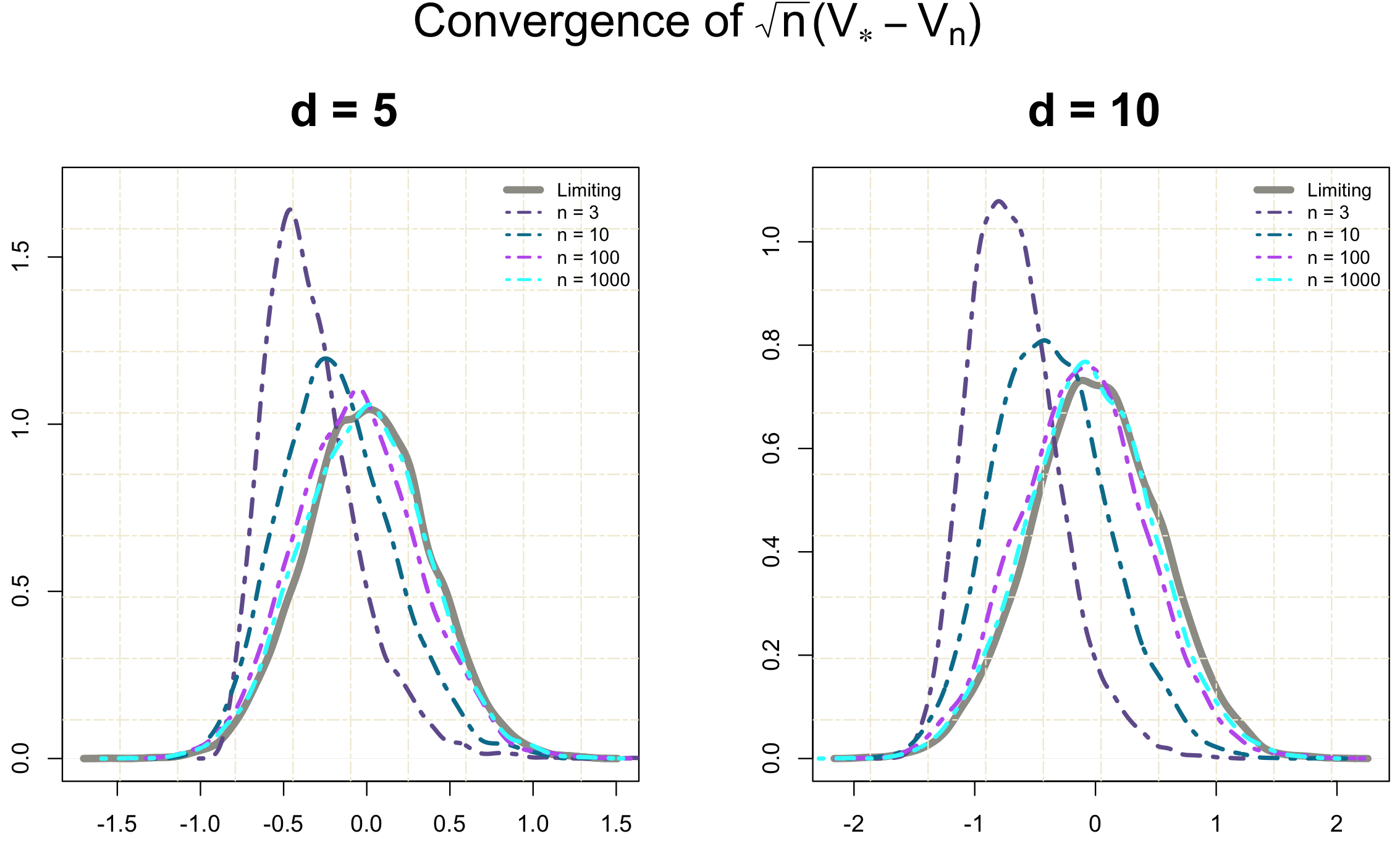}
	\caption{Densities of \(\mathcal{L}\left(\sqrt{n}(\mathcal{V}_n - \mathcal{V}_*)\right)\)}
	\label{fig:var}
\end{figure}

\subsection{Data aggregation in climate modelling}
In this section we carry out the experiments
on a family of Gaussian process, using a climate-related data set, 
collected in Siberia (Russia) between 1930 and 2009 \cite{bulygina2012daily, tatusko1990cooperation}.
We set \(\mathcal{SL}(\mu)\) to be
a family of Gaussian curves,
that describe the daily minimum temperatures within one year,
measured at a set of 30 randomly sampled meteorological stations.  
Each curve is obtained by means of 
regression and maximum likelihood estimation 
and is sampled in \(50\) points.
More details on this data set are provided in~\cite{fragen2017learning}.
The scale-location family under consideration is written as
\[
\mathcal{SL}(\mu) = \left\{\nu_t = GP(m_t, S_t),~
m_t \in \R^{50}, ~S_t \in \sym_{++}(50), ~t = 1,..., 71
\right\},
\]
where \(\nu_t = GP(m_t, S_t)\)
is a Gaussian process, characterised by mean \(m_t\) and covariance \(S_t\) inherent
to a year \(t\), \(t \in \{ 1933,...,  2009\} \). 
We let \(m_t = 0\) for all \(t\).
A Gaussian process \(\mu_* = GP(r_*, Q_*)\) is the population Wasserstein barycenter of \(\mathcal{SL}(\mu)\). It is characterised by
\((r_*, Q_*) \)
\[
r_* = 0,
\quad
Q_* = \frac{1}{71}\sum^{71}_{t = 1}\left(Q^{1/2}_* S_t Q^{1/2}_* \right)^{1/2}.
\]
A family of approximating processes\(\mu_n = GP(r_n, Q_n) \) with parameters
\((r_n, Q_n)\) \eqref{def:rn} is constructed 
by means of re-sampling with replacement 
of the original data set.
Sample size \(n\) varies in range \(n = \{3, 10, 70\}\).
Fig.~\ref{fig:F_GP} and Fig.~\ref{fig:BW_GP} present densities of \(\mathcal{L}\left( \sqrt{n}\norm{Q_n - Q_*}_{F}\right)\) and \(\mathcal{L}\left( \sqrt{n}d_{BW}(Q_n, Q_*)\right)\) respectively.
\begin{figure}
  \begin{minipage}[b]{0.45\textwidth}
    \includegraphics[width=\textwidth]{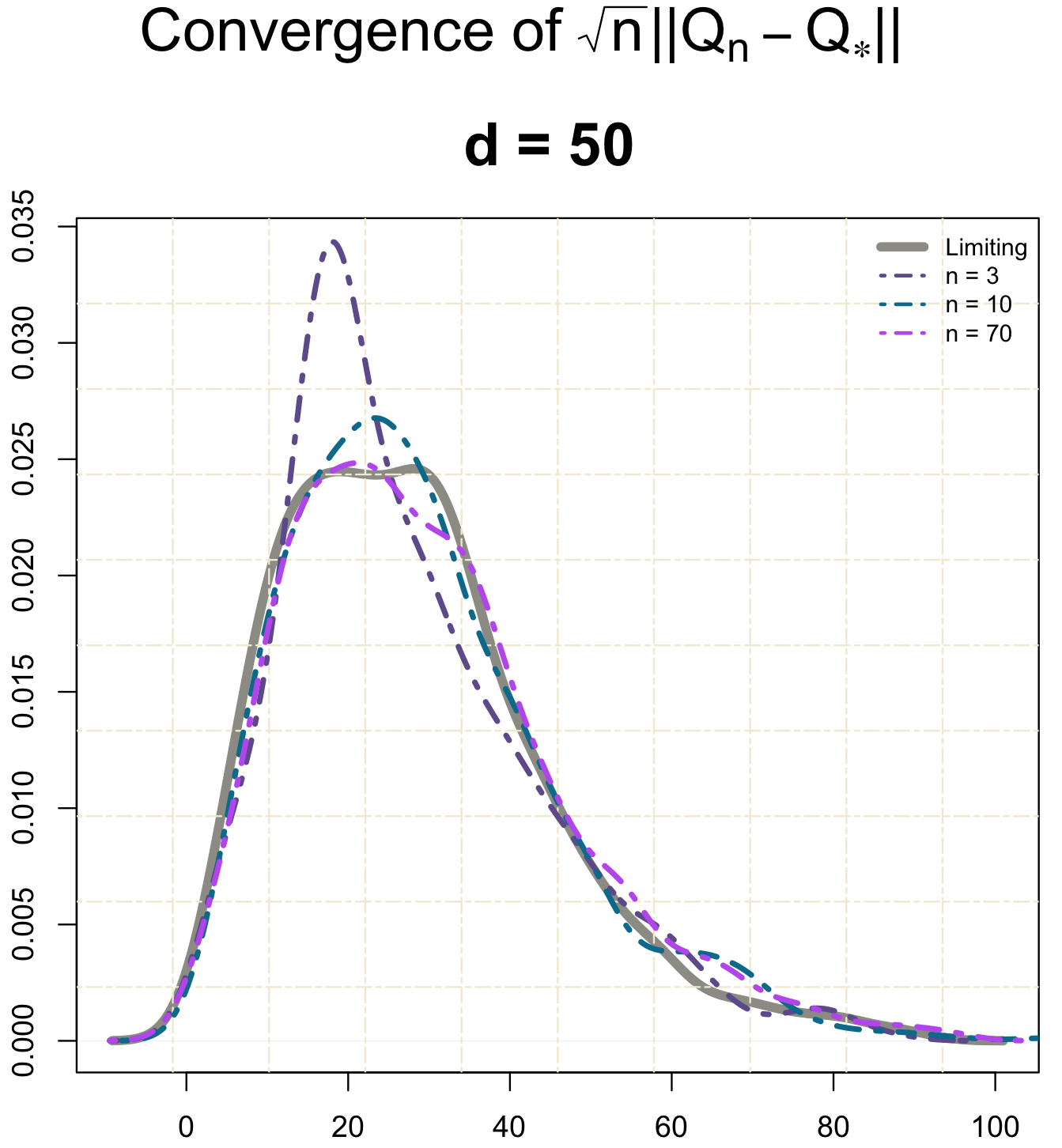}
    \caption{\(\mathcal{L}\left(\sqrt{n}\norm{Q_n - Q_*}_{F} \right)\)}
    \label{fig:F_GP}
  \end{minipage}
  \begin{minipage}[b]{0.45\textwidth}
    \includegraphics[width=\textwidth]{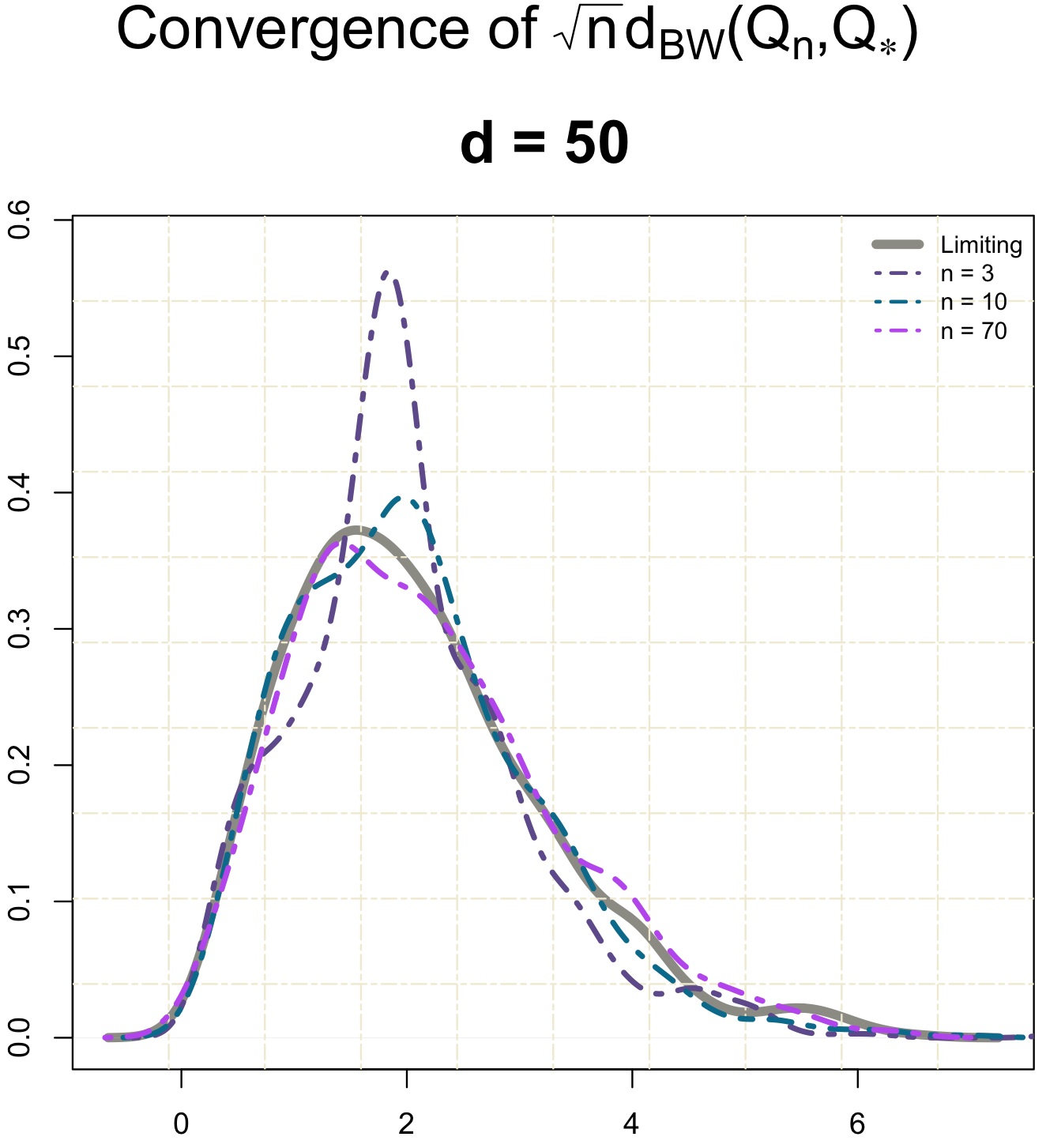}
    \caption{\(\mathcal{L}\left(\sqrt{n}d_{BW}(Q_n, Q_*) \right)\)}
    \label{fig:BW_GP}
  \end{minipage}
\end{figure}

\vskip 0.2in
\bibliographystyle{plainnat}
\bibliography{references.bib}

\begin{thebibliography}{34}
\providecommand{\natexlab}[1]{#1}
\providecommand{\url}[1]{\texttt{#1}}
\expandafter\ifx\csname urlstyle\endcsname\relax
  \providecommand{\doi}[1]{doi: #1}\else
  \providecommand{\doi}{doi: \begingroup \urlstyle{rm}\Url}\fi

\bibitem[Agueh and Carlier(2011)]{agueh2011barycenters}
Martial Agueh and Guillaume Carlier.
\newblock Barycenters in the {W}asserstein space.
\newblock \emph{SIAM Journal on Mathematical Analysis}, 43\penalty0
  (2):\penalty0 904--924, 2011.

\bibitem[Ahidar-Coutrix et~al.(2018)Ahidar-Coutrix, Gouic, and
  Paris]{ahidar2018rate}
Adil Ahidar-Coutrix, Thibaut~Le Gouic, and Quentin Paris.
\newblock On the rate of convergence of empirical barycentres in metric spaces:
  curvature, convexity and extendible geodesics.
\newblock \emph{arXiv preprint arXiv:1806.02740}, 2018.

\bibitem[{{\'A}lvarez-Esteban} et~al.(2015){{\'A}lvarez-Esteban}, {del Barrio},
  {Cuesta-Albertos}, and {Matr{\'a}n}]{alvarez2011wide}
P.~C. {{\'A}lvarez-Esteban}, E.~{del Barrio}, J.~A. {Cuesta-Albertos}, and
  C.~{Matr{\'a}n}.
\newblock Wide consensus for parallelized inference.
\newblock \emph{ArXiv e-prints}, November 2015.

\bibitem[Alvarez-Esteban et~al.(2018)Alvarez-Esteban, del Barrio,
  Cuesta-Albertos, Matr{\'a}n, et~al.]{alvarez2018wide}
Pedro~C Alvarez-Esteban, Eustasio del Barrio, Juan~A Cuesta-Albertos, Carlos
  Matr{\'a}n, et~al.
\newblock Wide consensus aggregation in the {W}asserstein space. application to
  location-scatter families.
\newblock \emph{Bernoulli}, 24\penalty0 (4A):\penalty0 3147--3179, 2018.

\bibitem[Ambrosio and Gigli(2013)]{ambrosio2013user}
Luigi Ambrosio and Nicola Gigli.
\newblock A user’s guide to optimal transport.
\newblock In \emph{Modelling and optimisation of flows on networks}, pages
  1--155. Springer, 2013.

\bibitem[Bhatia et~al.(2018)Bhatia, Jain, and Lim]{bhatia2018bures}
Rajendra Bhatia, Tanvi Jain, and Yongdo Lim.
\newblock On the {B}ures--{W}asserstein distance between positive definite
  matrices.
\newblock \emph{Expositiones Mathematicae}, 2018.

\bibitem[Bigot et~al.(2012)Bigot, Klein, et~al.]{bigot2012consistent}
J{\'e}r{\'e}mie Bigot, Thierry Klein, et~al.
\newblock Consistent estimation of a population barycenter in the {W}asserstein
  space.
\newblock \emph{ArXiv e-prints}, 2012.

\bibitem[Bigot et~al.(2017)Bigot, Cazelles, and Papadakis]{bigot2017penalized}
J{\'e}r{\'e}mie Bigot, Elsa Cazelles, and Nicolas Papadakis.
\newblock Penalized barycenters in the {W}asserstein space.
\newblock 2017.

\bibitem[Brenier(1991)]{brenier1991polar}
Yann Brenier.
\newblock Polar factorization and monotone rearrangement of vector-valued
  functions.
\newblock \emph{Communications on pure and applied mathematics}, 44\penalty0
  (4):\penalty0 375--417, 1991.

\bibitem[Bulygina and Razuvaev(2012)]{bulygina2012daily}
ON~Bulygina and VN~Razuvaev.
\newblock Daily temperature and precipitation data for 518 {R}ussian
  meteorological stations.
\newblock Technical report, ESS-DIVE (Environmental System Science Data
  Infrastructure for a Virtual Ecosystem); Oak Ridge National Lab.(ORNL), Oak
  Ridge, TN (United States), 2012.

\bibitem[Calsbeek and Goodnight(2009)]{calsbeek2009empirical}
Brittny Calsbeek and Charles~J Goodnight.
\newblock Empirical comparison of {G} matrix test statistics: finding
  biologically relevant change.
\newblock \emph{Evolution: International Journal of Organic Evolution},
  63\penalty0 (10):\penalty0 2627--2635, 2009.

\bibitem[Dajka et~al.(2011)Dajka, {\L}uczka, and H{\"a}nggi]{dajka2011distance}
Jerzy Dajka, Jerzy {\L}uczka, and Peter H{\"a}nggi.
\newblock Distance between quantum states in the presence of initial
  qubit-environment correlations: A comparative study.
\newblock \emph{Physical Review A}, 84\penalty0 (3):\penalty0 032120, 2011.

\bibitem[{Del Barrio} et~al.(2016){Del Barrio}, {Gordaliza}, {Lescornel}, and
  {Loubes}]{barrio2016central}
E.~{Del Barrio}, P.~{Gordaliza}, H.~{Lescornel}, and J.-M. {Loubes}.
\newblock Central limit theorem and bootstrap procedure for {W}asserstein's
  variations with application to structural relationships between
  distributions.
\newblock \emph{ArXiv e-prints}, November 2016.

\bibitem[del Barrio et~al.(2017)del Barrio, Cuesta-Albertos, Matr{\'a}n, and
  Mayo-{\'I}scar]{del2017robust}
E~del Barrio, JA~Cuesta-Albertos, C~Matr{\'a}n, and A~Mayo-{\'I}scar.
\newblock Robust clustering tools based on optimal transportation.
\newblock \emph{Statistics and Computing}, pages 1--22, 2017.

\bibitem[Del~Barrio et~al.(2015)Del~Barrio, Lescornel, and
  Loubes]{del2015statistical}
Eustasio Del~Barrio, H{\'e}l{\`e}ne Lescornel, and Jean-Michel Loubes.
\newblock A statistical analysis of a deformation model with {W}asserstein
  barycenters: estimation procedure and goodness of fit test.
\newblock \emph{arXiv preprint arXiv:1508.06465}, 2015.

\bibitem[Fano(1957)]{fano1957description}
Ugo Fano.
\newblock Description of states in quantum mechanics by density matrix and
  operator techniques.
\newblock \emph{Reviews of Modern Physics}, 29\penalty0 (1):\penalty0 74, 1957.

\bibitem[Flamary et~al.(2018)Flamary, Cuturi, Courty, and
  Rakotomamonjy]{flamary2018wasserstein}
R{\'e}mi Flamary, Marco Cuturi, Nicolas Courty, and Alain Rakotomamonjy.
\newblock Wasserstein discriminant analysis.
\newblock \emph{Machine Learning}, 107\penalty0 (12):\penalty0 1923--1945,
  2018.

\bibitem[Gonzalez et~al.(2017)Gonzalez, Pasi, Petkevi\v{c}i\={u}t\.{e},
  Glowacki, and Maddocks]{gonzalez2017absolute}
Oscar Gonzalez, Marco Pasi, Daiva Petkevi\v{c}i\={u}t\.{e}, Jaroslaw Glowacki,
  and JH~Maddocks.
\newblock Absolute versus relative entropy parameter estimation in a
  coarse-grain model of {DNA}.
\newblock \emph{Multiscale Modeling \& Simulation}, 15\penalty0 (3):\penalty0
  1073--1107, 2017.

\bibitem[Goodnight and Schwartz(1997)]{goodnight1997bootstrap}
Charles~J Goodnight and James~M Schwartz.
\newblock A bootstrap comparison of genetic covariance matrices.
\newblock \emph{Biometrics}, pages 1026--1039, 1997.

\bibitem[Gramfort et~al.(2015)Gramfort, Peyr{\'e}, and
  Cuturi]{gramfort2015fast}
Alexandre Gramfort, Gabriel Peyr{\'e}, and Marco Cuturi.
\newblock Fast optimal transport averaging of neuroimaging data.
\newblock In \emph{International Conference on Information Processing in
  Medical Imaging}, pages 261--272. Springer, 2015.

\bibitem[Hsu et~al.(2012)Hsu, Kakade, Zhang, et~al.]{hsu2012tail}
Daniel Hsu, Sham Kakade, Tong Zhang, et~al.
\newblock A tail inequality for quadratic forms of subgaussian random vectors.
\newblock \emph{Electronic Communications in Probability}, 17, 2012.

\bibitem[Jozsa(1994)]{jozsa1994fidelity}
Richard Jozsa.
\newblock Fidelity for mixed quantum states.
\newblock \emph{Journal of modern optics}, 41\penalty0 (12):\penalty0
  2315--2323, 1994.

\bibitem[Koltchinskii et~al.(2011)]{koltchinskii2011neumann}
Vladimir Koltchinskii et~al.
\newblock Von {N}eumann entropy penalization and low-rank matrix estimation.
\newblock \emph{The Annals of Statistics}, 39\penalty0 (6):\penalty0
  2936--2973, 2011.

\bibitem[Kroshnin(2018)]{kroshnin2017fr}
Alexey Kroshnin.
\newblock Fr{\'e}chet barycenters in the {M}onge--{K}antorovich spaces.
\newblock \emph{Journal of Convex Analysis}, 25\penalty0 (4):\penalty0
  1371--1395, 2018.

\bibitem[Le~Gouic and Loubes(2015)]{gouic_existence_2015}
Thibaut Le~Gouic and Jean-Michel Loubes.
\newblock Existence and consistency of {W}asserstein barycenters.
\newblock June 2015.

\bibitem[Mallasto and Feragen(2017)]{fragen2017learning}
Anton Mallasto and Aasa Feragen.
\newblock Learning from uncertain curves: The 2-{W}asserstein metric for
  {G}aussian processes.
\newblock In I.~Guyon, U.~V. Luxburg, S.~Bengio, H.~Wallach, R.~Fergus,
  S.~Vishwanathan, and R.~Garnett, editors, \emph{Advances in Neural
  Information Processing Systems 30}, pages 5660--5670. Curran Associates,
  Inc., 2017.

\bibitem[Marian and Marian(2008)]{marian2008bures}
Paulina Marian and Tudor~A Marian.
\newblock Bures distance as a measure of entanglement for symmetric two-mode
  {G}aussian states.
\newblock \emph{Physical Review A}, 77\penalty0 (6):\penalty0 062319, 2008.

\bibitem[Montavon et~al.(2016)Montavon, M{\"u}ller, and
  Cuturi]{montavon2016wasserstein}
Gr{\'e}goire Montavon, Klaus-Robert M{\"u}ller, and Marco Cuturi.
\newblock Wasserstein training of restricted {B}oltzmann machines.
\newblock In \emph{Advances in Neural Information Processing Systems}, pages
  3718--3726, 2016.

\bibitem[Muzellec and Cuturi(2018)]{muzellec2018generalizing}
Boris Muzellec and Marco Cuturi.
\newblock Generalizing point embeddings using the {W}asserstein space of
  elliptical distributions.
\newblock \emph{arXiv preprint arXiv:1805.07594}, 2018.

\bibitem[Rippl et~al.(2016)Rippl, Munk, and Sturm]{rippl2016limit}
Thomas Rippl, Axel Munk, and Anja Sturm.
\newblock Limit laws of the empirical {W}asserstein distance: {G}aussian
  distributions.
\newblock \emph{Journal of Multivariate Analysis}, 151:\penalty0 90--109, 2016.

\bibitem[Takatsu et~al.(2011)]{takatsu2011wasserstein}
Asuka Takatsu et~al.
\newblock Wasserstein geometry of {G}aussian measures.
\newblock \emph{Osaka Journal of Mathematics}, 48\penalty0 (4):\penalty0
  1005--1026, 2011.

\bibitem[Tatusko(1990)]{tatusko1990cooperation}
Rene Tatusko.
\newblock Cooperation in climate research: An evaluation of the activities
  conducted under the {US--USSR} agreement for environmental protection since
  1974.
\newblock 1990.

\bibitem[Villani(2009)]{villani_optimal_2009}
Cédric Villani.
\newblock \emph{Optimal Transport}, volume 338 of \emph{Grundlehren der
  mathematischen Wissenschaften}.
\newblock Springer Berlin Heidelberg, 2009.
\newblock ISBN 978-3-540-71049-3 978-3-540-71050-9.

\bibitem[Wassermann et~al.(2010)Wassermann, Bloy, Kanterakis, Verma, and
  Deriche]{wassermann2010unsupervised}
Demian Wassermann, Luke Bloy, Efstathios Kanterakis, Ragini Verma, and Rachid
  Deriche.
\newblock Unsupervised white matter fiber clustering and tract probability map
  generation: Applications of a {G}aussian process framework for white matter
  fibers.
\newblock \emph{NeuroImage}, 51\penalty0 (1):\penalty0 228--241, 2010.

\end{thebibliography}
\appendix

\section{Proof of Central Limit Theorem}
\label{sec:proof}


\subsection{Properties of $T_Q^S$}
\label{sec:proof_frechet}

\begin{proof}[Proof of Lemma~\ref{lemma:def_T}]
    First, we prove that optimal \(T\) is self-adjoint. Indeed, assume the opposite, then 
    \[
    Q^{1/2} T Q T^* Q^{1/2} = \left(Q^{1/2} T Q^{1/2}\right) \left(Q^{1/2} T Q^{1/2}\right)^* = Q^{1/2} S Q^{1/2}
    \]
    and thus \(\tr Q^{1/2} T Q^{1/2} < \tr \left(Q^{1/2} S Q^{1/2}\right)^{1/2}\). 
    Therefore
    \begin{align*}
        \tr (T - I) Q (T^* - I) 
        & = \tr S + \tr Q - 2 \tr T Q 
        = \tr S + \tr Q - 2 \tr Q^{1/2} T Q^{1/2} \\
        & > \tr S + \tr Q - 2 \tr \left(Q^{1/2} S Q^{1/2}\right)^{1/2} = d_{BW}^2(Q, S).
    \end{align*}
    
    If \(T\) is Hermitian but not positive semi-definite, then \(Q^{1/2} T Q^{1/2} \preccurlyeq \left(Q^{1/2} S Q^{1/2}\right)^{1/2}\), \(Q^{1/2} T Q^{1/2} \neq \left(Q^{1/2} S Q^{1/2}\right)^{1/2}\), hence again \(\tr Q^{1/2} T Q^{1/2} < \tr \left(Q^{1/2} S Q^{1/2}\right)^{1/2}\). 
    
    Finally, if \(T \in \H_+(n)\), then it is straightforward to check that \(T = T_Q^S\) given by~\eqref{def:T_Q} and
    \[
    \tr (T - I) Q (T^* - I) 
    = \tr S + \tr Q - 2 \tr \left(Q^{1/2} S Q^{1/2}\right)^{1/2}
    = d_{BW}^2(Q, S).
    \]
\end{proof}

The proof of the Central Limit theorem mainly relies on 
the differentiability of the map~\eqref{def:T_Q}. 
Lemma~\ref{lemma:taylor_T} shows that 
\(T_Q^S\)
can be linearised in the vicinity of \(Q\):
\[
T_{Q + X}^S 
= T_Q^S + \dT{Q}{S}(X) + o\bigl(\norm{X}\bigr),   
\]
where \(\dT{Q}{S} \colon \H(d) \to \H(d)\) is a self-adjoint negative-definite operator
and \(\norm{\cdot}\) stands for an operator norm. 
Properties of \(\dT{Q}{S}\) are investigated in Lemma~\ref{lemma:dT}. 
Let us introduce some notation: if \(G(A)\) is a functional of a matrix \(A\), then
we denote its differential as \(\o{d_A}G\).

\begin{lemma}
\label{lemma:taylor_root}
    Map \(Q \mapsto g(Q) = Q^{1/2}\) is differentiable on \(\H_{++}(d)\), and its differential is given by
    \[
    \o{d_Q}g(X) = U^* \left(\frac{(U X U^*)_{i j}}{\sqrt{q_i} + \sqrt{q_j}}\right)_{i, j = 1}^d U, \quad
    X \in \H(d),
    \]
    where \(Q = U^* \diag(q) U\) is the eigenvalue decomposition.
\end{lemma}
\begin{proof}
    First, let us consider the map \(P \mapsto f(P) = P^2\). It is smooth and its differential
    \[
    \o{d_P}f(X) = P X + X P, \quad
    X \in \H(d)
    \]
    is non-degenerated:
    \[
    \langle \o{d_P}f(X), X \rangle = 2 \tr X P X > 0, \quad
    X \neq 0,
    \]
    whenever \(P \in \H_{++}(d)\). From now on \(\langle \cdot, \cdot\rangle\) denotes a scalar product associated to Frobenius norm.
    
    Now applying the inverse function theorem we obtain that the inverse map \(g(\cdot)\) is also smooth and its differential enjoys the following equation
    \[
    X = \left(\left.\o{d_P}f\right|_{P = Q^{1/2}}\right)\left(\o{d_Q}g(X)\right) = 
    Q^{1/2} \o{d_Q}g(X) + \o{d_Q}g(X) Q^{1/2},
    \]
    thus
    \[
    U X U^* = (\diag(q))^{1/2} U \o{d_Q}g(X) U^* + U \o{d_Q}g(X) U^* (\diag(q))^{1/2},
    \]
    \[
    (U X U^*)_{i j} = (\sqrt{q_i} + \sqrt{q_j}) (U \o{d_Q}g(X) U^*)_{i j}, \quad 1 \le i, j \le d,
    \]
    and
    \[
    \o{d_Q}g(X) = U^* \left(\frac{(U X U^*)_{i j}}{\sqrt{q_i} + \sqrt{q_j}}\right)_{i, j = 1}^d U.
    \]
\end{proof}

\begin{lemma}[Fr{\'e}chet-differentiability of the map \(T^S_Q\)]
\label{lemma:taylor_T}
    For any \(S \in \H_{+}(d)\) the map \(T_{Q}^{S}\) can be linearised in the vicinity of \(Q \in \H_{++}(d)\) as
    \[
    T_{\tilde{Q}}^{S} 
    = T_{Q}^{S} + \dT{Q}{S}\left(\tilde{Q} - Q\right) + o\left(\norm{\tilde{Q} - Q}\right), \quad
    \text{as} \quad \tilde{Q} \to Q,
    \]
    where 
    \begin{equation}
    \label{eq:DT}
        \dT{Q}{S}(X) \eqdef - S^{1/2} U^* \Lambda^{-1/2} \delta \Lambda^{-1/2} U S^{1/2}, \quad
        X \in \H(d),
    \end{equation}
    \(U^* \Lambda U\) is an eigenvalue decomposition of \(S^{1/2} Q S^{1/2}\)
    \begin{gather*}
        U^* \Lambda U = S^{1/2} Q S^{1/2}, \quad 
        U^* U = U U^* = I, \quad
        \Lambda = \diag\left(\lambda_1, \dots, \lambda_{\rank(S)}, 0, \dots, 0\right), \\
        \Lambda^{-1/2} = \left(\Lambda^{1/2}\right)^+ = \diag(\lambda_1^{-1/2}, \dots, \lambda_{\rank(S)}^{-1/2}, 0, \dots, 0), \\
        \delta = (\delta_{i j})_{i, j = 1}^d, \quad
        \delta_{i j} = 
        \begin{cases}
            \frac{\Delta_{i j}}{\sqrt{\lambda_i} + \sqrt{\lambda_j}}, & i, j \le \rank(S)\\
            0, & \text{ otherwise}
        \end{cases}, \quad 
        \Delta = U S^{1/2} X S^{1/2} U^*.
    \end{gather*}
\end{lemma}
\begin{proof}
    The proof mainly relies on the differentiation of the pseudo-inverse term
    \(\bigl(S^{1/2} Q S^{1/2} \bigr)^{-1/2} \), 
    as soon as
    \[
    \dT{Q}{S}(X) = S^{1 / 2} \o{d_Q} \left(S^{1 / 2} Q S^{1 / 2}\right)^{- 1 / 2}(X) S^{1 / 2}.
    \]
    Obviously we can consider only restriction to $\range(S)$ and therefore assume w.l.o.g. $S \succ 0$.
    As \(\left(S^{1 / 2} (Q + X) S^{1 / 2}\right)^{- 1 / 2} = U^* \left(\Lambda + \Delta\right)^{- 1 / 2} U\), by  Lemma~\ref{lemma:taylor_root} and von Neumann series expansion we obtain for infinitesimal \(X \in \H(d)\) and corresponding \(\Delta\) that
    \begin{align*}
        \left(\Lambda + \Delta\right)^{- 1 / 2}
        & = \left(\Lambda^{1/2} + \delta + o(\norm{\Delta})\right)^{-1} \\
        & = \Bigl(\Lambda^{1 / 4} \left(I + \Lambda^{-1 / 4} \delta \Lambda^{-1 / 4} + o(\norm{\Delta}) \right) \Lambda^{1 / 4} \Bigr)^{-1} \\
        & = \Lambda^{- 1 / 4} \left(I - \Lambda^{-1 / 4} \delta \Lambda^{-1 / 4} + o(\norm{\Delta}) \right) \Lambda^{-1/4} \\
        & = \Lambda^{- 1 / 2} - \Lambda^{-1 / 2} \delta \Lambda^{-1 / 2} + o(\norm{\Delta}).
    \end{align*}

    Then the differential \(\o{d_{Q}}\left(S^{1/2} Q S^{1/2}\right)^{-1/2}(X)\) is written as
    \begin{equation*}
        \o{d_{Q}}\left(S^{1/2} Q S^{1/2}\right)^{-1/2}(X) 
        = - U^* \Lambda^{-1/2} \delta \Lambda^{-1/2} U.
    \end{equation*}
    Therefore,
    \[
    T^S_{Q + X} = T^S_{Q} + \dT{Q}{S}(X) + o(\norm{X}),
    \]
    where \(\dT{Q}{S}(X)\) is defined by~\eqref{eq:DT}.
\end{proof}

Lemmas~\ref{lemma:dT} and~\ref{lemma:integral_form} are technical and explore properties of \(\dT{Q}{S}\). 
 
\begin{lemma}
\label{lemma:dT}
    For any \(S \in \H_+(d)\), \(Q \in \H_{++}(d)\), the properties of operator \(\dT{Q}{S}\) defined in~\eqref{eq:DT} 
    are following:
    \begin{enumerate}[(I)]
        \item it is self-adjoint;
        
        \item it is negative semi-definite;
        
        \item it enjoys the following bounds:
        \begin{align*}
            - \left\langle \dT{Q}{S}(X), X \right\rangle &
            \le \frac{\lmax^{1/2}\left(S^{1/2} Q S^{1/2}\right)}{2} \norm{Q^{- 1 / 2} X Q^{- 1 / 2}}_F^2,\\
            - \left\langle \dT{Q}{S}(X), X \right\rangle &
            \ge \frac{\lmin^{1/2}\left(S^{1/2} Q S^{1/2}\right)}{2} \norm{Q^{- 1 / 2} X Q^{- 1 / 2}}_F^2;
        \end{align*}

        \item it is homogeneous w.r.t.\ $Q$ with degree $-\frac{3}{2}$ and w.r.t.\ $S$ with degree $\frac{1}{2}$, i.e.\ \(\dT{a Q}{S} = a^{-3 / 2} \dT{Q}{S}\) and \(\dT{Q}{a S} = a^{1 / 2} \dT{Q}{S}\) for any $a > 0$;
        
        \item it is monotone w.r.t.\ $S^{1/2} Q S^{1/2}$ (once range $S$ is fixed): \(\dT{Q_0}{S_0} \preccurlyeq \dT{Q_1}{S_1}\) in the sense of self-adjoint operators on $\H(d)$ whenever \(S_0^{1/2} Q_0 S_0^{1/2} \preccurlyeq S_1^{1/2} Q_1 S_1^{1/2}\) and \(\range(S_0) = \range(S_1)\); in particular, $\dT{Q}{S}$ is monotone w.r.t.\ $Q \in \H_{++}(d)$ for fixed $S$.
    \end{enumerate}
\end{lemma}
\begin{proof}
    Slightly changing notations, we rewrite~\eqref{eq:DT} as
    \[
    \dT{Q}{S}(X) = - S^{1 / 2} U^* \Lambda^{-1 / 2} \delta^{X} \Lambda^{-1 / 2} U S^{1 / 2},
    \]
    where matrices \(U, U^*\) and \(\Lambda\) come from Lemma~\ref{lemma:taylor_T} and
    \[
    \delta^X = (\delta^X_{i j})_{i, j = 1}^d, \quad
    \delta^X_{i j} = \frac{\Delta^X_{i j}}{\sqrt{\lambda_i} + \sqrt{\lambda_j}}, \quad
    \Delta^X = U S^{1 / 2} X S^{1 / 2} U^*.
    \]
    
    \subsubsection*{(I) Self-adjointness}
    Consider a scalar product 
    \begin{align*}
        \langle \dT{Q}{S}(X), Y \rangle 
        = \tr\bigl(\dT{Q}{S}(X) Y \bigr) 
        & = - \tr\bigl(S^{1/2} U^* \Lambda^{-1/2} \delta^X \Lambda^{-1/2} U S^{1/2} Y \bigr) \\
        {} & = - \tr\bigl(\Lambda^{-1/2} \delta^{X} \Lambda^{-1/2} U S^{1/2} Y S^{1/2} U^* \bigr).
    \end{align*}
    We now introduce a following notation
    \[
    \Delta^{Y} \eqdef U S^{1/2} Y S^{1/2} U^*.
    \]
    Then the above equality can be continued as follows:
    \begin{align*}
        - & \tr\bigl(\Lambda^{-1/2} \delta^{X} \Lambda^{-1/2} U S^{1/2} Y S^{1/2} U^* \bigr)
        = - \tr\bigl( \Lambda^{-1/2} \delta^{X} \Lambda^{-1/2} \Delta^{Y} \bigr) \\
        {} & \qquad = - \sum_{i, j = 1}^r \frac{\delta^{X}_{i j}}{\sqrt{\lambda_i \lambda_j}} \Delta^{Y}_{i j} 
        = - \sum_{i, j = 1}^r \frac{\Delta^{X}_{i j} \Delta^{Y}_{i j}}{\sqrt{\lambda_i \lambda_j}(\sqrt{\lambda_i} + \sqrt{\lambda_j})} \\
        {} & \qquad = \tr\bigl(\dT{Q}{S}(Y) X \bigr)
        = \tr\bigl(X \dT{Q}{S}(Y)\bigr)
        = \langle X, \dT{Q}{S}(Y) \rangle,
    \end{align*}
    where \(r = \rank(S)\).
    Thus the operator is self-adjoint.
    
    \subsubsection*{(II) Boundedness and (III) eigenvalues}
    Denoting \(\Delta^X\) by \(\Delta\) (i.e. now \(\Delta = U S^{1 / 2} X S^{1 / 2} U^*\)) and taking into account the above expansion of an inner product, one obtains
    \begin{align}
    \label{eq:scalar_product}
        - \langle \dT{Q}{S}(X), X \rangle 
        = \sum_{i, j = 1}^r \frac{\Delta^2_{i j}}{\sqrt{\lambda_i \lambda_j}(\sqrt{\lambda_i} + \sqrt{\lambda_j})} 
        = \sum_{i, j = 1}^r \left(\frac{\Delta_{i j}}{\sqrt{\lambda_i \lambda_j}}\right)^2 \frac{\sqrt{\lambda_i \lambda_j}}{\sqrt{\lambda_i} + \sqrt{\lambda_j} }.
    \end{align}
    
    Note, that the function
    \(
    f(\lambda_i, \lambda_j) \eqdef \frac{\sqrt{\lambda_i \lambda_j}}{\sqrt{\lambda_i} + \sqrt{\lambda_j}}
    \)
    is monotonously increasing in both arguments \(\lambda_i\) and \(\lambda_j \), thus
    \begin{equation}
    \label{eq:eigen_values_delta}
        \max_{i, j} f(\lambda_i, \lambda_j) = \frac{\lmax^{1/2}(\Lambda)}{2}, \quad
        \min_{i, j} f(\lambda_i, \lambda_j) = \frac{\lmin^{1/2}(\Lambda)}{2}.
    \end{equation}
    For the sake of simplicity we introduce a new variable
    \[
    \zeta \eqdef Q^{- 1 / 2} X Q^{- 1 / 2},
    \]
    its Frobenius norm is written as
    \[
    \norm{\zeta}_F^2
    = \tr \bigl(X Q^{- 1} X Q^{- 1} \bigr).
    \]
    Moreover, the following inequality for trace holds:
    \begin{align*}
        \tr \bigl(X Q^{- 1} X Q^{- 1} \bigr) 
        & \ge \tr \bigl(\Pi_S X \Pi_S Q^{- 1} \Pi_S X \Pi_S Q^{- 1} \Pi_S \bigr) \\
        & = \tr \bigl(\Delta \Lambda^{+} \Delta \Lambda^{+} \bigr)
        = \norm{\Lambda^{- 1 / 2} \Delta \Lambda^{- 1 / 2}}_F^2
        = \sum_{i, j = 1}^r \frac{\Delta_{i j}^2}{\lambda_i \lambda_j}.
    \end{align*}
    Here \(\Pi_S\) is the orthogonal projector onto range of \(S\).
  
Then combining~\eqref{eq:scalar_product} with~\eqref{eq:eigen_values_delta}, 
    the upper and lower bounds can be obtained as follows: 
    \begin{align*}
        -\langle \dT{Q}{S}(X), X \rangle 
        & \le \max_{i, j} f(\lambda_i, \lambda_j) \sum_{i, j = 1}^r \left(\frac{\Delta_{ij}}{\sqrt{\lambda_i \lambda_j}}\right)^2
        \le \frac{\lmax^{1/2}(\Lambda)}{2} \norm{\zeta}_F^2,\\
        -\langle \dT{Q}{S}(X), X \rangle 
        & \ge \min_{i, j} f(\lambda_i, \lambda_j) \sum_{i, j = 1}^r \left(\frac{\Delta_{ij}}{\sqrt{\lambda_i \lambda_j}}\right)^2
        = \frac{\lmin^{1/2}(\Lambda)}{2} \norm{\zeta}_F^2.
    \end{align*}
    Note, that if \(S\) is degenerated, the lower bound becomes trivial.

    \subsubsection*{{(IV) Homogeneity and (V) monotonicity}}
    Homogeneity follows directly from definition~\eqref{eq:DT}. 
    Now we prove monotonicity. As range of \(S^{1/2} Q S^{1/2}\) is fixed, we may assume \(S \succ 0\).
    Consider
    \begin{align*}
        \langle \dT{Q}{S}(X), X \rangle 
        & = \tr \left( S^{1/2} U^* \Lambda^{-1/2} \delta \Lambda^{-1/2} U S^{1/2}, X \right) \\
        & = \left\langle U^* \Lambda^{-1/2} \delta \Lambda^{-1/2} U, S^{1/2} X S^{1/2} \right\rangle \\
        & = \left\langle \o{d_{Q}}\left(S^{1/2} Q S^{1/2}\right)^{-1/2}(X), S^{1/2} X S^{1/2} \right\rangle \\
        & = \left\langle \o{d_{M}} M^{-1/2}\left(S^{1/2} X S^{1/2}\right), S^{1/2} X S^{1/2} \right
        \rangle,
    \end{align*}
    with replacement
    \(M = S^{1/2} Q S^{1/2}\) to be change of variables.
    As soon as~\(X\) is supposed to be fixed, 
    it is enough to show that the differential \(\o{d_M} M^{-1 / 2}\) is monotone in~\(M\).
    Notice that the operator \(\left(\o{d_M} M^{-1 / 2}\right)^{-1}\) at point~$M$ is equal to the differential of the inverse map \(P \mapsto P^{-2}\) at point \(P = M^{-1 / 2}\):
    \[
    \o{d_M} M^{-1 / 2}
    = \left(\left.\o{d_P} P^{-2}\right|_{P = M^{-1/2}}\right)^{-1}.
    \]
    In turn,
    \(\o{d_P} P^{-2}\) can be expressed as
    \[
    \o{d_P} P^{-2}(X) = - P^{-1} \left(P^{-1} X + X P^{-1}\right) P^{-1},
    \]
    the right part of the above equation is self-adjoint, negative-definite and
    \[
    \left\langle - P^{-1} \left(P^{-1} X + X P^{-1}\right) P^{-1}, X \right\rangle 
    = - 2 \tr P^{-2} X P^{-1} X.
    \]
    Choose \(M_1 \succcurlyeq M_0 \succ 0\) (thus \(M_1^{1 / 2} \succcurlyeq M_0^{1 / 2}\)) and let \(P_i = M_i^{-1 / 2}\) for \(i = 0, 1\). 
    Then for any fixed \(X \in \H(d)\)
    \[
    - \tr P_1^{-2} X P_1^{-1} X 
    = - \tr M_1 X M_1^{1 / 2} X 
    \le - \tr M_0 X M_0^{1 / 2} X
    = - \tr P_0^{-2} X P_0^{-1} X, 
    \]
    i.e.\ 
    \(\left.\o{d_P} P^{-2}\right|_{P_1} \preccurlyeq \left.\o{d_P} P^{-2}\right|_{P_0}\) 
    and hence for the differential of \(M \mapsto M^{-1 / 2}\) the inverse inequality holds: 
    \(\left.\o{d_M} M^{- 1 / 2}\right|_{M_0} \preccurlyeq \left.\o{d_M} M^{- 1 / 2}\right|_{M_1}\). This entails monotonicity of \(\dT{Q}{S}\).
\end{proof}

\begin{corollary}
\label{corollary:dt}
    We define a following rescaled operator
    \begin{equation}
    \label{def:dt}
        \dt{Q}{S}(\zeta) \eqdef Q^{1/2} \dT{Q}{S}\left(Q^{1/2} \zeta Q^{1/2}\right) Q^{1/2}, \quad \zeta \in \H(d).
    \end{equation}
    Then a following bound on its eigenvalues hold:
    \begin{align*}
        & \lmin\left(-\dt{Q}{S}\right) = \frac{1}{2} \lmin^{1/2}\left(S^{1/2} Q S^{1/2}\right), \\
        & \lmax\left(-\dt{Q}{S}\right) = \frac{1}{2} \lmax^{1/2}\left(S^{1/2} Q S^{1/2}\right).
    \end{align*}
\end{corollary}

\begin{proof}
    Notice that inequalities
    \begin{align*}
        & \lmin\left(-\dt{Q}{S}\right) \ge \frac{1}{2} \lmin^{1/2}\left(S^{1/2} Q S^{1/2}\right), \\
        & \lmax\left(-\dt{Q}{S}\right) \le \frac{1}{2} \lmax^{1/2}\left(S^{1/2} Q S^{1/2}\right),
    \end{align*}
    are a trivial consequence of Lemma~\ref{lemma:dT} (III). 
    Now defining for any \(1 \le k \le \rank(S)\)
    \[
    \Delta_{i j}^k = 
    \begin{cases}
        1, & i = j = k,\\
        0, & \text{otherwise},
    \end{cases}, \quad 
    X^k = S^{-1/2} U \Delta^k U^* S^{-1/2}, \quad
    \zeta^k = Q^{-1/2} X^k Q^{-1/2}
    \]
    we obtain from~\eqref{eq:scalar_product} that
    \[
    - \left\langle \dt{Q}{S}(\zeta^k), \zeta^k \right\rangle
    = - \left\langle \dT{Q}{S}(X^k), X^k \right\rangle
    = \frac{\lambda_k^{1/2}}{2} \norm{\zeta^k}_F^2.
    \]
    Therefore, the above inequalities are sharp.
\end{proof}

    

\begin{lemma}
\label{lemma:integral_form}
    For any \(Q_0, Q_1 \in \H_{++}(d)\), \(S \in \H_+(d)\) consider
    \begin{equation}
    \label{def:Q'}
        Q_t \eqdef (1 - t) Q_0 + t Q_1, 
        \quad
        Q' \eqdef Q_0^{-1/2} Q_1 Q_0^{-1/2}.
    \end{equation}
    Then
    \begin{align*}
    \tag{I}
        \frac{2}{\lmin(Q') + {\lmin^{1/2}(Q')}} \dT{Q_0}{S}
        & \preccurlyeq \int_0^1 \dT{Q_t}{S} \diff t \\ 
        & \preccurlyeq \frac{2}{\lmax(Q') + {\lmax^{1/2}(Q')}} \dT{Q_0}{S} \\
        & \preccurlyeq \frac{1}{1 + 3 \norm{Q' - I} / 4} \dT{Q_0}{S}.
    \end{align*}
    Moreover, if \(\norm{Q' - I} < 1\), then
    \begin{equation*}
    \tag{II}
    \label{eq:dT_bounds}
        \int_0^1 \dT{Q_t}{S} \diff t
        \succcurlyeq \frac{1}{1 - \norm{Q' - I}} \dT{Q_0}{S}.
    \end{equation*}
\end{lemma}

\begin{remark}
    The above inequality might seem confusing due to the fact that \(\lmin(\cdot) \leq \lmax(\cdot)\),
    however this is explained by the fact that \(\dT{Q}{S}\) is \emph{negative} definite.
\end{remark}

\begin{proof}
    Notice that
    \[
    Q_t = Q_0^{1/2} \bigl((1 - t) I + t Q'\bigr) Q_0^{1 / 2}.
    \]
    Monotonicity and homogeneity with degree \(-\frac{3}{2}\) 
    of \(\dT{Q}{S}(\cdot)\) (see Lemma~\ref{lemma:dT}) yield
    \begin{align*}
        \dT{Q_t}{S}
        & \preccurlyeq \dT{((1 - t) + t \lmax(Q')) Q_0}{S} \\
        & = \bigl((1 - t) + t \lmax(Q')\bigr)^{- 3 / 2} \dT{Q_0}{S}
    \end{align*}
    and 
    \begin{align*}
        \dT{Q_t}{S}
        & \succcurlyeq \dT{((1 - t) + t \lmin(Q')) Q_0}{S} \\
        & = \bigl((1 - t) + t \lmin(Q')\bigr)^{- 3 / 2} \dT{Q_0}{S}.
    \end{align*}
    Therefore,
    \begin{align*}
        \int_0^1 \dT{Q_t}{S} \diff t
        & \preccurlyeq \dT{Q_0}{S} \int_0^1 \bigl((1 - t) + t \lmax(Q')\bigr)^{- 3 / 2} \diff t \\
        {} & = \frac{2}{\lmax(Q') + {\lmax^{1/2}(Q')}} \dT{Q_0}{S}
    \end{align*}
    and respectively,
    \[
    \int_0^1 \dT{Q_t}{S} \diff t
    \succcurlyeq \frac{2}{\lmin(Q') + {\lmin^{1/2}(Q')}} \dT{Q_0}{S}.
    \]
    The rest of the Lemma follow from the fact that
    \[
    \lmin(Q') \ge 1 - \norm{Q' - I}, \quad
    \lmax(Q') \le 1 + \norm{Q' - I}
    \]
    and inequalities
    \begin{gather*}
        \sqrt{1 + x} \le 1 + \frac{x}{2} \text{~~for~~} x \ge 0,\\
        \sqrt{1 - x} \ge 1 - x \text{~~for~~} 0 \le x \le 1.
    \end{gather*}
\end{proof}

\subsection{Properties of \(d_{BW}(Q, S)\)}
\label{sec:propeties_dbw_proof}
The next lemma ensures strict convexity of \(d_{BW}(Q, S)\).
In essence, the proof mainly relies on Theorem 7~\cite{bhatia2018bures}.
\begin{lemma}
\label{lemma:strict_convexity}
    For any \(S \in \H_{+}(d)\) function \(Q \mapsto d_{BW}^2(Q, S)\) is convex on \(\H_{+}(d)\). Moreover, if \(S \succ 0\), then it is strictly convex.
\end{lemma}
\begin{proof}
    According to \citep[Theorem 7]{bhatia2018bures} function \(h(X) = \tr X^{1/2}\) is strictly concave on \(\H_{+}(d)\), hence function
    \[
    Q \mapsto d_{BW}^2(Q, S) 
    = \tr S + \tr Q - 2 \tr \left(S^{1/2} Q S^{1/2}\right)^{1/2}
    \]
    is convex on \(\H_{+}(d)\) for any positive semi-definite \(S\). Moreover, if \(S \succ 0\), then \(Q \mapsto S^{1/2} Q S^{1/2}\) is an injective linear map, and therefore \(d_{BW}^2(Q, S)\) is strictly convex.
\end{proof}

Further we present differentiability of \(d_{BW}(Q, S)\) and its quadratic approximation.
\begin{lemma}
\label{lemma:quadratic_approx}
    For any \(Q \in \H_{++}(d)\), \(S \in \H_+(d)\) function \(d_{BW}^2(Q, S)\) is twice differentiable in $Q$ with
    \begin{align*}
        & \o{d_Q} d_{BW}^2(Q, S) (X) = \langle I - T_Q^S, X \rangle, & X \in \H(d),\\
        & \o{d^2_Q} d_{BW}^2(Q, S) (X, Y) = - \langle X, \dT{Q}{S}(Y) \rangle, 
        & X, Y \in \H(d).
    \end{align*}
    Moreover, the following quadratic approximation holds: for any \(Q_0, Q_1 \in \H_{++}(d)\)
    \begin{align*}
        - \tfrac{2}{\left(1 + \lmax^{1/2}(Q')\right)^2} & \left\langle \dT{Q_0}{S}(Q_1 - Q_0), Q_1 - Q_0 \right\rangle \\
        & \le d_{BW}^2(Q_1, S) - d_{BW}^2(Q_0, S) + \langle T_{Q_0}^S - I, Q_1 - Q_0 \rangle \\
        & \le - \tfrac{2}{\left(1 + \lmin^{1/2}(Q')\right)^2} \left\langle \dT{Q_0}{S}(Q_1 - Q_0), Q_1 - Q_0 \right\rangle.
    \end{align*}
    with \(Q'\) defined in~\eqref{def:Q'}.
\end{lemma}
\begin{proof}
    Note that
    \[
    \o{d_Q} \left(S^{1/2} Q S^{1/2}\right)^{1/2} (X) 
    = U^* \delta U,
    \]
    where \(\delta\) comes from Lemma~\ref{lemma:taylor_T}. 
    Furthermore, Lemma~\ref{lemma:taylor_root} implies that
    \begin{align*}
        \o{d_Q} \tr \left(S^{1/2} Q S^{1/2}\right)^{1/2} (X)
        & = \tr \o{d_Q} \left(S^{1/2} Q S^{1/2}\right)^{1/2} (X) 
        = \tr \delta \\
        & = \sum_{i = 1}^{\rank(S)} \frac{\Delta_{i i}}{2 \sqrt{\lambda_i}} 
        = \frac{1}{2} \tr \Delta \Lambda^{-1/2} \\
        & = \frac{1}{2} \tr S^{1/2} X S^{1/2} \left(S^{1/2} Q S^{1/2}\right)^{-1/2}
        = \frac{1}{2} \left\langle T_Q^S, X \right\rangle.
    \end{align*}
    Consequently, \(d_{BW}^2(Q, S)\) is differentiable and
    \begin{align*}
        \o{d_Q} d_{BW}^2(Q, S) (X) 
        = \tr X - 2 \o{d_Q} \tr \left(S^{1/2} Q S^{1/2}\right)^{1/2} (X) 
        = \left\langle I - T_Q^S, X \right\rangle.
    \end{align*}
    Applying Lemma~\ref{lemma:taylor_T} one obtains
    \[
    \o{d^2_Q} d_{BW}^2(Q, S) (X, Y) 
    = \o{d_Q} \left\langle I - T_Q^S, X \right\rangle
    = - \left\langle \dT{Q}{S}(Y), X \right\rangle (Y).
    \]
    
    \paragraph{Quadratic approximation}
    Let \(Q_0, Q_1 \in \H_{++}(d)\), \(Q_t \eqdef (1 - t) Q_0 + t Q_1\), \(t \in [0, 1]\).
    The Taylor expansion in integral form applied to \(d_{BW}^2(Q_t, S)\) implies
    \begin{align*}
        d_{BW}^2(Q_1, S) 
        & = d_{BW}^2(Q_0, S) + \left\langle I - T_{Q_0}^S, Q_1 - Q_0 \right\rangle\\
        & \quad \quad + \int_0^1 (1 - t) \left\langle -\dT{Q_t}{S}(Q_1 - Q_0), Q_1 - Q_0 \right\rangle \diff t \\
        & = d_{BW}^2(Q_0, S) - \left\langle T_{Q_0}^S - I, Q_1 - Q_0 \right\rangle \\
        & \quad \quad - \left\langle \left[\int_0^1 (1 - t) \dT{Q_t}{S}\diff t \right] (Q_1 - Q_0), Q_1 - Q_0 \right\rangle.
    \end{align*}
    Following the same ideas as in the proof of Lemma~\ref{lemma:integral_form} one obtains that
    \begin{align*}
        \int_0^1 (1 - t) \dT{Q_t}{S} \diff t 
        & \preccurlyeq \int_0^1 (1 - t) \bigl((1 - t) + t \lmax(Q')\bigr)^{- 3 / 2} \dT{Q_0}{S} \diff t \\
        & = \tfrac{2}{\left(1 + \lmax^{1/2}(Q')\right)^2} \dT{Q_0}{S}
    \end{align*}
    and 
    \[
    \int_0^1 (1 - t) \dT{Q_t}{S} \diff t 
    \succcurlyeq \tfrac{2}{\left(1 + \lmin^{1/2}(Q')\right)^2} \dT{Q_0}{S}.
    \]
    Thus
    \begin{align*}
        - \tfrac{2}{\left(1 + \lmax^{1/2}(Q')\right)^2} & \left\langle \dT{Q_0}{S}(Q_1 - Q_0), Q_1 - Q_0 \right\rangle \\
        & \le d_{BW}^2(Q_1, S) - d_{BW}^2(Q_0, S) + \langle T_{Q_0}^S - I, Q_1 - Q_0 \rangle \\
        & \le - \tfrac{2}{\left(1 + \lmin^{1/2}(Q')\right)^2} \left\langle \dT{Q_0}{S}(Q_1 - Q_0), Q_1 - Q_0 \right\rangle.
    \end{align*}
\end{proof}

\subsection{Central limit theorem for \(Q_n\)}
\label{section:CLT_proof}

First let us prove uniqueness and positive-definiteness of Bures--Wasserstein barycenter.
\begin{proof}[Proof of Theorem~\ref{thm:uniqueness}]
    Since 
    \[
    \E d_{BW}^2(0, S) = \E \tr S < \infty
    \]
    and \(d_{BW}(Q, S) \to \infty\) as \(\norm{Q} \to \infty\), the barycenter \(Q_*\) always exists by continuity and compactness argument.
    In case \(\P(\H_{++}(d)) > 0\) applying Lemma~\ref{lemma:strict_convexity} we obtain strict convexity of the integral
    \[
    Q \mapsto \E d_{BW}^2(Q, S) = \mathcal{V}(Q), \quad Q \in \H_+(d),
    \]
    and therefore, uniqueness of the minimizer \(Q_*\). 

    To prove that \(Q_* \succ 0\) consider arbitrary degenerated \(Q_0 \in \H_+(d) \cap \A\), \(Q_1 \in \H_{++}(d) \cap \A\) (which exists by Assumption~\ref{asm:main}) and \(S \in \H_{++}(d)\). Let us define \(Q_t = (1 - t) Q_0 + t Q_1 \in \A\). We are going to show, that 
    \[
    \frac{d}{d t} d_{BW}^2(Q_t, S) = \langle I - T_{Q_t}^S, Q_1 - Q_0 \rangle \to - \infty \quad \text{as} \quad t \to 0.
    \]
    Consider eigen-decomposition \(S^{1/2} Q_0 S^{1/2} = U^* \Lambda U\), \(\Lambda = \diag(\lambda_1, \dots, \lambda_r, 0, \dots, 0)\), where \(r = \rank(Q_0)\). Respectively, we write \(C = U S^{1/2} Q_1 S^{1/2} U^*\) in a block form 
    \[
    C = 
    \begin{pmatrix}
        C_{1 1} & C_{1 2} \\
        C_{2 1} & C_{2 2}
    \end{pmatrix}, \quad
    C_{1 1} \in \H_{++}(r), \;
    C_{1 2} = C_{2 1}^* \in \C^{r \times (d - r)}, \;
    C_{2 2} \in \H_{++}(d - r).
    \]
    Cramer's rule for inverse matrix and Laplace's formula yield
    \[
    U \left(S^{1/2} Q_t S^{1/2}\right)^{-1} U^*
    = \bigl((1 - t) \Lambda + t C\bigr)^{-1}
    = \begin{pmatrix}
        O(1) & O(1) \\
        O(1) & C_{2 2}^{-1} / t + O(1)
    \end{pmatrix}, \text{ as } t \to 0,
    \]
    therefore
    \[
    \sqrt{t} U \left(S^{1/2} Q_t S^{1/2}\right)^{-1/2} U^* \to 
    \begin{pmatrix}
        0 & 0 \\
        0 & C_{2 2}^{-1/2}
    \end{pmatrix}.
    \]
    Consequently,
    \begin{align*}
        \sqrt{t} \left\langle T_{Q_t}^S, Q_0 \right\rangle 
        & = \sqrt{t} \left\langle \left(S^{1/2} Q_t S^{1/2}\right)^{-1/2}, S^{1/2} Q_0 S^{1/2} \right\rangle \\
        & = \left\langle \sqrt{t} U \left(S^{1/2} Q_t S^{1/2}\right)^{-1/2} U^*, U S^{1/2} Q_0 S^{1/2} U^* \right\rangle \\
        & = \left\langle 
        \begin{pmatrix}
            o(1) & o(1) \\
            o(1) & C_{2 2}^{-1/2}
        \end{pmatrix}, 
        \begin{pmatrix}
            \diag(\lambda_1, \dots, \lambda_r) & 0 \\
            0 & 0
        \end{pmatrix} 
        \right\rangle 
        = o(1).
    \end{align*}
    In the same way one can obtain
    \begin{align*}
        \sqrt{t} \left\langle T_{Q_t}^S, Q_1 \right\rangle 
        = \left\langle 
        \begin{pmatrix}
            o(1) & o(1) \\
            o(1) & C_{2 2}^{-1/2}
        \end{pmatrix},
        \begin{pmatrix}
            C_{1 1} & C_{1 2} \\
            C_{2 1} & C_{2 2}
        \end{pmatrix} 
        \right\rangle
        = \tr C_{2 2}^{1/2} + o(1).
    \end{align*}
    Consequently,
    \[
    \frac{d}{d t} d_{BW}^2(Q_t, S) 
    = \langle I - T_{Q_t}^S, Q_1 - Q_0 \rangle
    = \tr Q_1 - \tr Q_0 - \frac{\tr C_{2 2}^{1/2} + o(1)}{\sqrt{t}} 
    \to - \infty.
    \]
    Since \(\P(\H_{++}(d)) > 0\) by Assumption~\ref{asm:main} and \(d_{BW}^2(Q, S)\) is convex, we conclude that
    \[
    \frac{d}{d t} \mathcal{V}(Q_t) 
    = \E \frac{d}{d t} d_{BW}^2(Q_t, S) 
    \to - \infty \quad \text{as} \quad t \to 0,
    \]
    thus \(Q_0\) cannot be a barycenter of $\P$. This yields \(Q_* \succ 0\).

    Since $\mathcal{V}(\cdot)$ is convex and barycenter of $\P$ is positive-definite and unique, it is characterized as a stationary point of Fr\'echet variation on subspace $\A$, i.e.\ as a solution to equation
    \[
    \o{\Pi}_\M \nabla \mathcal{V}(Q) = \o{\Pi}_\M (I - \E T_Q^S) = 0, \quad Q \in \A \cap \H_{++}(d),
    \]
    as required.
\end{proof}


The proof of CLT widely uses 
the concept of covariance operators
on the space of optimal transportation maps
and on the space of covariance matrices.

\paragraph{Covariance operator on the space of optimal maps}
Consider \(T_i \eqdef T_{Q_*}^{S_i}\) and \(T^n_i \eqdef T_{Q_n}^{S_i}\).
We define covariance \(\o{\Sigma}\) of \(T_i\), its empirical counterpart \(\o{\Sigma_n}\), 
and its data-driven estimator \(\o{\hat{\Sigma}_n}\)
as follows
\begin{gather}
    \label{def:covariance}
    \o{\Sigma} \eqdef \var(T_i), 
    \quad
    \o{\Sigma_n} \eqdef \frac{1}{n} \sum_{i = 1}^n \left(T_i - I\right) \otimes \left(T_i - I\right), \nonumber \\
    \o{\hat{\Sigma}_n} \eqdef \frac{1}{n} \sum_{i = 1}^n \left(T^n_i - I\right) \otimes \left(T^n_i - I\right).
\end{gather}

\paragraph{Covariance operators on the space of covariance matrices}
Let \(Q_n\) be an empirical barycenter.
The covariance of \(Q_n\) and its empirical counterpart
are defined as
\begin{equation}
\label{def:Upxi}
    \o{\Upxi} \eqdef \o{F}^{-1} (\o{\Sigma})_\M \o{F}^{-1},
    \quad
    \o{\Upxi} \colon \M \rightarrow \M,
\end{equation}
\begin{equation}
\label{def:Upxi_n}
    \o{\hat{\Upxi}_n} \eqdef \o{\hat{F}_n}^{-1} (\o{\hat{\Sigma}_n})_\M \o{\hat{F}_n}^{-1}, \quad
    \o{\hat{\Upxi}_n} \colon \M \rightarrow \M,
\end{equation}
where
\begin{gather}
    \label{def:F}
    \o{F} \eqdef - \E \left(\dT{Q_*}{S}\right)_\M \quad
    \o{F_n} \eqdef - \frac{1}{n} \sum_{i = 1}^n \left(\dT{Q_*}{S_i}\right)_\M,\\
    \label{def:Fnhat}
    \o{\hat{F}_n} \eqdef - \frac{1}{n} \sum_{i = 1}^n \left(\dT{Q_n}{S_i}\right)_\M.
\end{gather}

Now we are ready to prove the central limit theorem for the
empirical barycenter \(Q_n\) (Theorem~\ref{thm:CLT}).
However, for the sake of transparency we provide below
a complete statement.
\begin{theorem*}[Central limit theorem for the covariance of empirical barycenter]
\label{thm:CLT_complete}
    The approximation error rate of the Fr\'echet mean \(Q_*\)
    by its empirical counterpart \(Q_n\) is
    \begin{equation}
    \label{eq:A}
        \tag{A} 
        \sqrt{n} \left(Q_n - Q_*\right) \rightharpoonup \ND\left(0, \o{\Upxi}\right),
    \end{equation}
    Moreover, {if \((\o{\Sigma})_\M\) is non-degenerated, then}
    \begin{equation}
    \label{eq:B}
        \tag{B}
        \sqrt{n} \o{\hat{\Upxi}}^{-1/2}_n \left(Q_n - Q_*\right) \rightharpoonup \ND\left(0, (\Id)_\M\right).
    \end{equation}
\end{theorem*}

\begin{proof}[Proof of Theorem~\ref{thm:CLT}]
The proof consists of two parts: proof of~\eqref{eq:A} and~\eqref{eq:B}.
\subsubsection*{Proof of~\eqref{eq:A}}
    As \(\mathcal{V}_n(\cdot)\) are convex functions, 
    they a.s.\ uniformly converge to strictly convex function \(\mathcal{V}(\cdot)\) on any compact set by the uniform law of large numbers. Therefore, their minimizers also converge a.s.\ \(Q_n \to Q_*\). In particular, \(Q_n \succ 0\) with dominating probability.

    Applying the expansion from Lemma~\ref{lemma:taylor_T} at point \(Q_*\) implies
    \begin{align}
    \label{eq:T_exp}
        T^n_i & = T_i + \int_0^1 \dT{Q_t}{S_i}(Q_n - Q_*) \diff t \\
        & = T_i + \dT{Q_*}{S_i}(Q_n - Q_*) + \alpha\bigl(Q_n - Q_*; Q_n, S_i\bigr), \nonumber
    \end{align}
    where \(Q_t = (1 - t) Q_* + t Q_n\) and \(\alpha(\cdot; Q_n, S_i)\) is a self-adjoint operator on \(\H(d)\) s.t.\ \(\left\langle X, \alpha\bigl(X; Q_n, S\bigr)\right \rangle = o\bigl(\bigl\langle X, \dT{Q_*}{S}(X) \bigr\rangle\bigr)\) as \(Q_n' \to I\) uniformly in \(S\) and \(X\) according to Lemma~\ref{lemma:integral_form}.
    Note, that the condition for \(Q_n\) being a barycenter 
    is \(\o{\Pi}_\M \left(\frac{1}{n} \sum_i T^n_i - I \right) = 0\).
    This fact together with averaging of~\eqref{eq:T_exp} over \(i\) give:
    \begin{equation}
    \label{eq:T_exp_2}
        \o{\Pi}_\M I = \o{\Pi}_\M \bar{T}_n
        - \o{F}_n (Q_n - Q_*)
        + \o{\Pi}_\M \alpha_n\bigl(Q_n - Q_*\bigr),
    \end{equation}
    where 
    \begin{equation}
        \label{def:mean_T}
        \bar{T}_n \eqdef \frac{1}{n} \sum_{i = 1}^n T_i,
        \quad
        \alpha_n(X) \eqdef \frac{1}{n}\sum^n_{i = 1}\alpha(X; Q_n, S_i),
    \end{equation}
    and \(\o{F}_n\) is defined in~\eqref{def:F}.
    Recall that \(\o{F}\)~\eqref{def:F} is a population counterpart of \(\o{F_n}\).
    This operator is correctly defined since by Lemma~\ref{lemma:dT}
    one can show that it is self-adjoint, positive definite and bounded:
    \[
    \norm{\o{F}} 
    \le \E \norm{\dT{Q_*}{S}}
    \le \E \frac{\norm{S^{1 / 2} Q_* S^{1 / 2}}}{2 \lmin^2(Q_*)} < \infty.
    \]
    
    Since by the law of large numbers \(\o{F}_n \to \o{F}\) and \(\alpha_n(X) = o\bigl(\norm{\o{F}_n(X)}\bigr)\), we obtain from~\eqref{eq:T_exp_2}
    \begin{align}
    \label{eq:decomp_Q}
        Q_n 
        = Q_* + \o{F}^{-1} \o{\Pi}_\M \left(\bar{T}_n - I\right)
        + o\left(\norm{\o{\Pi}_\M \left(\bar{T}_n - I\right)} \right),
    \end{align}
    where \(\o{F}^{-1}\) is a bounded linear operator, 
    because \(\dT{Q_*}{S}\) is negative definite for any \(S \succ 0\) by Lemma~\ref{lemma:dT}.
    The result~\eqref{eq:A} follows immediately from the CLT for \(\o{\Pi}_\M \bar{T}_n\).
    
    \subsubsection*{Proof of~\eqref{eq:B}}
    Note that result~\eqref{eq:A} is equivalent to the fact, that
    \[
    \sqrt{n}\o{\Upxi}^{-1/2}\left(Q_n - Q_* \right) \rightharpoonup \ND\left(0, (\o{Id})_\M \right).
    \]
    To ensure convergence of \(\o{\hat{\Upxi}_n} \rightarrow \o{\Upxi}\)
    we need to show that
    \begin{itemize}
        \item \(\o{\hat{\Sigma}_n} \rightarrow \o{\Sigma}\) (is proved in Lemma~\ref{lemma:covergence_Sigma});
        \item \(\o{\hat{F}_n} \rightarrow \o{F}\).
    \end{itemize}

    \paragraph{Convergence of \(\o{\hat{F}_n}\) to \(\o{F}\)}
    Monotonicity and homogeneity with degree \(-\frac{3}{2}\) 
    of \(\dT{Q}{S}\) (see Lemma~\ref{lemma:dT}) yield
    \[
    \bigl(\lmax(Q'_n)\bigr)^{- 3 / 2} \dT{Q_*}{S}
    \succcurlyeq \dT{Q_n}{S}
    \succcurlyeq \bigl(\lmin(Q'_n)\bigr)^{- 3 / 2} \dT{Q_*}{S},
    \]
    where \(Q'_n\) comes from~\eqref{def:Q'}.
    This naturally leads to the following relation:
    \[
    \frac{1}{\lmax^{3/2}(Q'_n)} (\Id)_\M
    \preccurlyeq \o{F_n}^{-1/2} \o{\hat{F}_n} \o{F_n}^{-1/2} 
    \preccurlyeq \frac{1}{\lmin^{3/2}(Q'_n)} (\Id)_\M.
    \]
    Note, that \(\lmin(Q'_n) \geq 1 - \norm{Q'_n - I}\) and \(\lmax(Q'_n) \leq 1 + \norm{Q'_n - I}\).
    This immediately implies
    \[
    \norm{\o{F^{-1/2}_n} \o{\hat{F}_n} \o{F^{-1/2}_n} - (\Id)_\M} 
    \leq \left(1 - \norm{Q'_n - I}\right)^{-3/2} - 1
    \rightarrow 0.
    \]
    Since \(\o{F_n} \to \o{F}\) this implies \(\o{\hat{F}_n} \to \o{F}\).
    
    The above results ensures validity of
    substitution \(\o{\Upxi}\) by \(\o{\hat{\Upxi}_n}\). This yields~\eqref{eq:B}.
\end{proof}

The asymptotic convergence results for \(d_{BW}(Q_n, Q_*)\) is 
a straightforward corollary of the above theorem. Here is the proof.
\begin{proof}[Proof of Corollary~\ref{thm:asymptotic}]
    Since \(Q_n \to Q_*\) a.s., Lemma~\ref{lemma:quadratic_approx} implies
    \[
    d_{BW}(Q_n, Q_*) = \frac{1 + o(1)}{2} \left\langle \dT{Q_*}{Q_*} (Q_n - Q_*), Q_n - Q_* \right\rangle.
    \]
    By Theorem~\ref{thm:CLT} \(\sqrt{n} (Q_n - Q_*)\) is asymptotically normal
    and centred, therefore
    \[
    \mathcal{L}\left(\sqrt{n} d_{BW}(Q_n, Q_*) \right) 
    \rightharpoonup 
    \mathcal{L}\left(\norm{Q^{1/2}_* \dT{Q_*}{Q_*}(Z)}_F\right).
    \]
    where \(Z \in \M \subset \H(d)\) and \(Z \sim \ND\left(0,  \o{\Upxi} \right)\).
   
    Note, that \(Q_n \rightarrow Q_*\),
    \(\o{\hat{\Upxi}}_n \rightarrow \o{\Upxi} \),
    and
    {\(\dT{Q_n}{Q_n} \rightarrow \dT{Q_*}{Q_*}\)}. 
    The latter result follows from Lemma~\ref{lemma:dT} (IV, V):
    \begin{align*}
        \dT{Q_n}{Q_n} \preccurlyeq \dT{\lmax(Q_n') Q_*}{\lmax(Q_n') Q_*} = \frac{1}{\lmax(Q_n')} \dT{Q_*}{Q_*} \to \dT{Q_*}{Q_*}, \\
        \dT{Q_n}{Q_n} \succcurlyeq \dT{\lmin(Q_n') Q_*}{\lmin(Q_n') Q_*} = \frac{1}{\lmin(Q_n')} \dT{Q_*}{Q_*} \to \dT{Q_*}{Q_*}.
    \end{align*}
    This yields
    \[
    \mathcal{L}\left(\norm{Q^{1/2}_n \dT{Q_n}{Q_n}(Z_n)}_F\right) 
    \rightharpoonup \mathcal{L}\left(\norm{Q^{1/2}_* \dT{Q_*}{Q_*}(Z)}_F\right),
    \]
    where \(Z_n \sim \ND\left(0, \o{\hat{\Upxi}}_n \right)\).
    This, in turn, entails
    \[
    d_w\Bigl(\mathcal{L}\left(\sqrt{n} d_{BW}(Q_n, Q_*)\right), 
    \mathcal{L}\left(\norm{Q^{1/2}_n \dT{Q_n}{Q_n}(Z_n) }_F\right)\Bigr) \rightarrow 0.
    \qedhere
    \]
\end{proof}

\section{Concentration of barycenters}
\label{sec:concentration_proof}

The next lemma is a key ingredient in the proof of
concentration result for \(Q_n\).
\begin{lemma}
\label{lemma:bound_Q'}
    Consider
    \begin{equation}
    \label{def:eta_n}
        \eta_n \eqdef \frac{1}{\lmin(\o{F'}_n)} \norm{Q_*^{1/2} \o{\Pi}_\M \left(\bar{T}_n - I\right) Q_*^{1/2}}_F
    \end{equation}
    where
    \begin{equation}
    \label{def:F'n}
        \o{F'}_n(X) \eqdef Q_*^{1/2} \o{F}_n \left(Q_*^{1/2} X Q_*^{1/2}\right) Q_*^{1/2} 
        \quad\text{for}\quad X \in \left\{Q_*^{-1/2} Y Q_*^{-1/2} \vert Y \in \M\right\}.
    \end{equation}
    Then
    \[
    \norm{Q_n' - I}_F \le \frac{\eta_n}{1 - \tfrac{3}{4} \eta_n }
    \]
    whenever \(\eta_n < \frac{4}{3}\) and \(Q_n \succ 0\).
\end{lemma}
\begin{proof}
    Let us define \(Q_t \eqdef t Q_n + (1 - t) Q_*\) for \(t \in [0, 1]\).
    Due to Lemmas~\ref{lemma:dT} and~\ref{lemma:integral_form} we have for any \(S \in \H_+(d)\)
    \begin{align*}
        \left\langle \o{\Pi}_\M \left(T_{Q_*}^S - T_{Q_n}^S\right), Q_n - Q_* \right\rangle 
        & = \left\langle T_{Q_*}^S - T_{Q_n}^S, Q_n - Q_* \right\rangle \\
        & = \int_0^1 \bigl\langle - \dT{Q_t}{S}(Q_n - Q_*), Q_n - Q_* \bigr\rangle \diff t \\ 
        & \ge \frac{1}{1 + \tfrac{3}{4} \norm{Q_n' - I} } \bigl\langle - \dT{Q_*}{S}(Q_n - Q_*), Q_n - Q_* \bigr\rangle.
    \end{align*}
    Therefore, 
    \begin{align*}
        \left\langle \o{\Pi}_\M \left(\bar{T}_n - I\right), Q_n - Q_* \right\rangle 
        & \ge \frac{1}{1 + \tfrac{3}{4}  \norm{Q_n' - I}} \bigl\langle \o{F}_n(Q_n - Q_*), Q_n - Q_* \bigr\rangle \\
        & = \frac{1}{1 +  \tfrac{3}{4}  \norm{Q_n' - I} } \bigl\langle \o{F'}_n(Q_n' - I), Q_n' - I \bigr\rangle \\
        & \ge \frac{\lmin(\o{F'}_n)}{1 + \tfrac{3}{4}  \norm{Q_n' - I} } \norm{Q_n' - I}_F^2.
    \end{align*}
    At the same time,
    \begin{align*}
    \left\langle \o{\Pi}_\M \left(\bar{T}_n - I\right), Q_n - Q_*\right\rangle 
    & = \left\langle Q_*^{1/2} \o{\Pi}_\M \left(\bar{T}_n - I\right) Q_*^{1/2}, Q_n' - I\right\rangle \\
    & \le \norm{Q_*^{1/2} \o{\Pi}_\M \left(\bar{T}_n - I\right) Q_*^{1/2}}_F \norm{Q_n' - I}_F.
    \end{align*}
    Hence
    \[
    \norm{Q_n' - I}_F 
    \le \frac{1 + \tfrac{3}{4} \norm{Q_n' - I} }{\lmin(\o{F'}_n)} \norm{Q_*^{1/2} \o{\Pi}_\M \left(\bar{T}_n - I\right) Q_*^{1/2}}_F
    = \left(1 + \tfrac{3}{4} \norm{Q_n' - I}\right) \eta_n.
    \]
    Rewriting the inequality above we obtain
    \[
    \norm{Q_n' - I}_F \le \frac{\eta_n}{1 - \tfrac{3}{4} \eta_n}
    \]
    provided that \(\eta_n < \frac{4}{3}\). 
\end{proof}

\begin{proposition}[Concentration of \(\bar{T}_n\); \cite{hsu2012tail}, Theorem 1]
\label{prop:concentr_T}
    Under Assumption~\ref{asm:subgauss} it holds
    \[
    \P\left\{\norm{\bar{T}_n - I}_F \ge \tfrac{\sigma_T}{\sqrt{n}} \left(d + t\right)\right\} \le e^{-t^2 / 2} \quad\text{for any} \quad t \ge 0.
    \]
\end{proposition}
Before proving concentration results, we define operator \(\o{F'}(X) \)
as:
\begin{equation}
    \label{def:F'}
            \o{F'}(X) \eqdef Q_*^{1/2} \o{F} \left(Q_*^{1/2} X Q_*^{1/2}\right) Q_*^{1/2} 
        \text{ for } X \in \left\{Q_*^{-1/2} Y Q_*^{-1/2} \vert Y \in \M\right\}.
\end{equation}

\begin{proof}[Proof of Theorem~\ref{theorem:concentration_Q}]
    Let \(t_n\) be s.t.\ the following upper bound on 
    \(\gamma_n(t_n)\) from Lemma~\ref{lemma:concentration_F}
    holds:
    \begin{equation}
    \label{def:gamma_n}
        \gamma_n(t_n) \leq \frac{1}{2} \lmin(\o{F'}).
    \end{equation}
    It is easy to see that this condition is fulfilled for $t_n = n t_F - \log(m)$ under a proper choice of generic constant in definition of $t_F$.
    Then with \(\P \ge 1 - 2 m e^{-n t_F}\) a following bound holds
    \[
    \lmin(\o{F'}_n) \ge \lmin(\o{F'}) - \norm{\o{F'}_n - \o{F'}} \ge \frac{1}{2} \lmin(\o{F'}).
    \]
   
    The above facts together with definition of \(\eta_n\)~\ref{def:eta_n} yield
    \begin{align*}
        \eta_n 
        & \eqdef \frac{\norm{Q_*^{1/2} \o{\Pi}_\M \left(\bar{T}_n - I\right) Q_*^{1/2}}_F}{\lmin(\o{F'}_n)} \\
        & \le \frac{2 \norm{Q_*}}{\lmin(\o{F'})} \norm{\o{\Pi}_\M \left(\bar{T}_n - I\right)}_F
        = \frac{c_Q}{2 \sigma_T} \norm{\o{\Pi}_\M \left(\bar{T}_n - I\right)}_F.
    \end{align*}
    Combining the above bounds with Proposition~\ref{prop:concentr_T}, 
    we obtain: 
    \[
    \P\left\{\eta_n \ge \frac{c_Q}{2 \sqrt{n}} (d + t)\right\}
    \le 2 m e^{-n t_F} + e^{-t^2 / 2}.
    \]
    Now it follows from Lemma~\ref{lemma:bound_Q'} that
    \begin{align*}
        \P\left\{\norm{Q_n' - I}_F \ge \frac{c_Q}{\sqrt{n}} (d + t)\right\}
        & \le \P\left\{2 \eta_n \ge \frac{c_Q}{\sqrt{n}} (d + t)\right\} + \P\bigl\{Q_n \nsucc 0\bigr\} \\
        & \le 2 m e^{-n t_F} + e^{-t^2 / 2} + (1 - p)^n
    \end{align*}
    whenever 
    \(\frac{c_Q}{2 \sqrt{n}} (d + t) \le \frac{2}{3}\). Here we used that \(Q_n \succ 0\) if at least one of matrices \(S_1, \dots, S_n\) is non-degenerated.
\end{proof}

\begin{proof}[Proof of Corollary~\ref{corollary:concantration_in_dbw}]
    To prove this result 
    we use Lemma~\ref{lemma:quadratic_approx}
    and choose \(Q_0 = S = Q_*\), \(Q_1 = Q_n\).
    Thus we obtain
    \begin{align*}
        d_{BW}^2(Q_n, Q_*) 
        & \le - \tfrac{2}{\left(1 + \lmin^{1/2}(Q'_n)\right)^2} \left\langle \dT{Q_*}{Q_*}(Q_n - Q_*), Q_n - Q_* \right\rangle \\
        & \overset{Def.~\ref{def:dt}}{=} \tfrac{2}{\left(1 + \lmin^{1/2}(Q'_n)\right)^2} \left\langle - \dt{Q_*}{Q_*}(Q'_n - I), Q'_n - I \right\rangle \\
        & \le 2 \lmax\left(-\dt{Q_*}{Q_*}\right) \norm{Q'_n - I}_F^2
        \overset{C.\ref{corollary:dt}}{\le} \lmax(Q_*) \norm{Q'_n - I}_F^2 
    \end{align*}
    Hence by Theorem~\ref{theorem:concentration_Q}
    \[
    d_{BW}(Q_n, Q_*) 
    \le \norm{Q_*}^{1/2} \frac{c_Q}{\sqrt{n}} (d + t)
    \]
    with probability at least \(1 - 2 m e^{- n t_F} - e^{-t^2 / 2} - (1 - p)^n\).
\end{proof}

\subsection{{Central limit theorem and concentration for \(\mathcal{V}_n\)}}
\label{sec:clt_V_proof}
\begin{proof}[Proof of Theorem~\ref{theorem:CLT_V}]
    By definition empirical Fr\'echet variance is
    \[
    \mathcal{V}_n(Q) = \frac{1}{n} \sum_{i = 1}^n d_{BW}^2(Q, S_i).
    \]
    Lemma~\ref{lemma:quadratic_approx} 
    ensures the following bound on \(\mathcal{V}_n(Q_*) - \mathcal{V}_n(Q_n)\):
    \[
    0 \le \mathcal{V}_n(Q_*) - \mathcal{V}_n(Q_n) 
    \le \tfrac{2}{\left(1 + \lmin^{1/2}(Q_n')\right)^2} \langle \o{F}_n(Q_n - Q_*), Q_n - Q_* \rangle
    \]
    with \(Q'_n\) defined in~\eqref{def:Q'_n}.
    The above quadratic bound
    together with  \(Q_n \to Q_*\), \(\o{F}_n \to \o{F}\) and \(\sqrt{n} (Q_n - Q_*) \rightharpoonup \ND(0, \o{\Upxi})\) yield:
    \[
    \mathcal{V}_n(Q_n) - \mathcal{V}(Q_*) = \mathcal{V}_n(Q_*) - \mathcal{V}(Q_*) + O\left(\frac{1}{n}\right).
    \]
    On the other hand, by classical central limit theorem we obtain:
    \begin{align*}
        \sqrt{n} \left(\mathcal{V}_n(Q_*) - \mathcal{V}(Q_*)\right) 
        & = \sqrt{n} \left(\frac{1}{n} \sum_i d_{BW}^2(Q_*, S_i) - \E d_{BW}^2(Q_*, S)\right) \\
        & \rightharpoonup \ND\left(0, \var d_{BW}^2(Q_*, S)\right).
    \end{align*}
\end{proof}

\begin{proof}[Proof of Theorem~\ref{theorem:concentration_V}]
    Following the proof of Theorem~\ref{theorem:CLT_V} we consider
    \(\mathcal{V}_n(Q_*) - \mathcal{V}_n(Q_n)\):
    \begin{align}
    \label{eq:bound_V}
        0 \le \mathcal{V}_n(Q_*) - \mathcal{V}_n(Q_n) 
        & \le \tfrac{2}{\left(1 + \lmin^{1/2}(Q_n')\right)^2} \langle \o{F}_n(Q_n - Q_*), Q_n - Q_* \rangle \nonumber \\
        & = \tfrac{2}{\left(1 + \lmin^{1/2}(Q_n')\right)^2} \langle \o{F'}_n(Q_n' - I), Q_n' - I \rangle \nonumber \\
        & \le 2 \norm{\o{F'}_n} \norm{Q_n' - I}_F^2.
    \end{align}
    Following the proof of Theorem~\ref{theorem:concentration_Q},
    we obtain that with \(\P \ge 1 - 2 m e^{-t_F n} - e^{-t^2/2} - (1 - p)^n\) the following upper bounds hold:
    \[
    \norm{Q'_n - I}_F \le \frac{c_Q}{\sqrt{n}} (d + t), \quad 
    \norm{\o{F'}_n - \o{F'}} \le \frac{1}{2} \lmin(\o{F'}).
    \]
    Thus
    \[
    \norm{\o{F'}_n} \le \norm{\o{F'}} + \norm{\o{F'}_n - \o{F'}} \le \frac{3}{2} \norm{\o{F'}}
    \]
    and consequently
    \[
    0 \le \mathcal{V}_n(Q_*) - \mathcal{V}_n(Q_n) \le 3 \norm{\o{F'}} \frac{c_Q^2 }{n} (d + t)^2.
    \]
    
    Now we consider a difference \(\mathcal{V}_n(Q_*) - \mathcal{V}(Q_*)\). 
    According to Assumption~\ref{asm:subgauss} \(S\) and therefore \(d_{BW}^2(Q_*, S)\) are sub-exponential r.v.
    {with some parameters \((\nu, b)\)}.
    Then Lemma~\ref{lemma:subexp} ensures
    \[
    \left|\mathcal{V}_n(Q_*) - \mathcal{V}(Q_*)\right| 
    \le \max\left(\frac{2 b t'}{n}, \nu \left(\frac{2 t'}{n}\right)^{1/2}\right)
    \]
    with probability \(1 - 2 e^{-t'}\).
    Combining two above bounds, we obtain:
    \[
    \left|\mathcal{V}_n(Q_n) - \mathcal{V}(Q_*) \right|
    \le \max\left(\frac{2 b t'}{n}, \nu \sqrt{\frac{2 t'}{n}}\right) + 3 \norm{\o{F'}} \frac{c_Q^2 }{n} (d + t)^2
    \]
    with probability
    \[
    \P \ge 1 - 2 e^{-t'} - 2 m e^{-n t_F} - e^{-t^2 / 2} - (1 - p)^n.
    \]
    Choosing \(t' = t^2/2\), we obtain
    \begin{multline*}
        \P\left\{\left|\mathcal{V}_n(Q_n) - \mathcal{V}(Q_*) \right|
        \ge \max\left(\frac{b t^2}{n}, \frac{\nu t}{\sqrt{n}}\right) + 3 \norm{\o{F'}} \frac{c_Q^2 }{n} (d + t)^2\right\} \\
        {} \le 2 m e^{-n t_F} + 3 e^{-t^2 / 2} + (1 - p)^n.
    \end{multline*}
\end{proof}

\subsection{Auxiliary results}
\label{section:approximation}

\begin{lemma}
\label{lemma:covergence_Sigma}
    Let \(\norm{Q_n' - I} \le \frac{1}{2}\); then
    \[
    \norm{\o{\hat{\Sigma}_n} - \o{{\Sigma}_n}}_1 
    \le \beta_n \left[2 \left(\frac{1}{n} \sum_i \norm{T_i - I}_F^2\right)^{1/2} + \beta_n \right],
    \]
    where
    \[
    \beta_n \eqdef \kappa(Q_*) \left(\frac{\frac{1}{n} \sum_i \norm{S_i}}{\norm{Q_*}}\right)^{1/2} \norm{Q'_n - I}_{F},
    \]
    where \(\kappa(Q_*) = {\norm{Q_*}}{\norm{Q_*^{-1}}}\) is the condition number of matrix \(Q_*\) and \(\norm{\cdot}_{1}\) is $1$-Schatten (nuclear) norm of an operator. 
\end{lemma} 
\begin{proof}
    Note, that for any 
    \(\left(T^{n}_{i} - I \right) \otimes \left(T^{n}_{i} - I \right)\)
    the following decomposition holds
    \begin{align*}
        \left(T^{n}_{i} - I \right) & \otimes \left(T^{n}_{i} - I \right) \\
        & = \left(T_i - I \right) \otimes \left(T_i - I \right)
        + \left(T^{n}_{i} - T_i \right) \otimes \left(T_i - I \right) \\
        & + \left(T_i - I \right) \otimes \left(T^{n}_{i} - T_i \right)
        + \left(T^{n}_{i} - T_i \right) \otimes \left(T^{n}_{i} - T_i \right).
    \end{align*}
    Summing by \(i\) yields
    \begin{align}
    \label{def:diff_covar}
        \o{\hat{\Sigma}_n} & - \o{{\Sigma}_n} 
        = \frac{1}{n} \sum_i \left(T^{n}_{i} - T_i \right) \otimes \left(T_i - I \right) \\ 
        & + \frac{1}{n} \sum_i \left(T_i - I \right) \otimes \left(T^{n}_{i} - T_i \right) 
        + \frac{1}{n} \sum_i \left(T^{n}_{i} - T_i \right) \otimes \left(T^{n}_{i} - T_i\right). \nonumber
    \end{align}
    
    Note, that each
    \[
    \norm{\left( T^{n}_{i} - T_i \right) \otimes \left(T_i - I \right)}_1
    \leq
    \norm{T^{n}_{i} - T_i}_F \norm{T_i - I}_F.  
    \]
    Lemmas~\ref{lemma:dT} (III) and~\ref{lemma:integral_form} yield
    \begin{align*}
        \norm{T^{n}_{i} - T_i}_F 
        & \leq \frac{1}{1 - \norm{Q_n' - I}} \norm{\dT{Q_*}{S_i}(Q_n - Q_*)}_F \\
        & \le 2 \norm{Q^{-1/2}_* \dt{Q_*}{S_i}\left(Q'_n - I\right) Q^{-1/2}_*}_F 
        \le 2 \frac{\lmax\left(\dt{Q_*}{S_i}\right)}{\lmin(Q_*)}  \norm{Q'_n - I}_F \\
        & \le \frac{\lmax^{1/2}\left(S_i^{1/2} Q_* S_i^{1/2}\right)}{\lmin(Q_*)} \norm{Q'_n - I}_F 
        \le \kappa(Q_*) \left(\frac{\norm{S_i}}{\norm{Q_*}}\right)^{1/2} \norm{Q'_n - I}_F,
    \end{align*}
    hence
    \(\frac{1}{n} \sum_i \norm{T^{n}_{i} - T_i}_F^2 \le \beta_n^2\).
    The above expression together with~\eqref{def:diff_covar} and Cauchy-Schwarz inequality
    lead to the upper bound on \(\norm{\o{\hat{\Sigma}_n} - \o{\Sigma_n}}_1\):
    \begin{align*}
        \norm{\o{\hat{\Sigma}_n} - \o{\Sigma_n}}_1
        & \le \frac{2}{n} \sum_i \norm{T_i - I}_F \norm{T^{n}_{i} - T_i}_F + \frac{1}{n} \sum_i \norm{T^{n}_{i} - T_i}_F^2 \\
        & \le 2 \beta_n \left(\frac{1}{n} \sum_i \norm{T_i - I}_F^2\right)^{1/2} + \beta_n^2.
    \end{align*}
\end{proof}

Further we present concentration of $\o{F}_n$ around $\o{F}$.
Denote as \(\norm{\cdot}_{\psi_2}\) an Orlicz norm with Young function 
\(\psi_2(x) = e^{x^2} - 1\), i.e.
\[
\norm{X}_{\psi_2} \eqdef \inf\bigl\{c > 0 : \E \psi_2\left(|X| / c\right) \le 1\bigr\}.
\]
Then sub-Gaussianity of a r.v.\ \(X\) is equivalent to \(\norm{X}_{\psi_2} < \infty\) and it ensures
\[
\var(X)\leq \sqrt{2} \norm{X}_{\psi_2}.
\]

\begin{lemma}[Concentration of $\o{F'}_n$, Proposition 2 in~\cite{koltchinskii2011neumann}]
\label{lemma:concentration_F}
    There exists a constant \(\CONST > 0\), s.t.
    for all \(t > 0\) it holds with 
    probability at least \(1 - e^{-t}\)
    \[
    \norm{\o{F'}_n - \o{F'}} 
    \leq \gamma_n(t), 
    \quad
    \gamma_n(t) \eqdef \CONST \max\left(\sigma_F \sqrt{\tfrac{t + \log(2 m)}{n}}, 
    U \sqrt{\log\left(\tfrac{U}{\sigma_F}\right)} \tfrac{t + \log(2 m)}{n} \right),
    \]
    where 
    \(\sigma_F \eqdef \norm{\E \left(\dt{Q_*}{S} - \o{F'}\right)^2}\), 
    \(U \eqdef \norm{\norm{\dt{Q_*}{S} - \o{F'}}}_{\psi_2}\).
\end{lemma}

\begin{lemma}[Sub-exponential tail bounds]
\label{lemma:subexp}
    Suppose that \(X\) is sub-exponential 
    with parameters \(\nu, b\). Then
    \[
    \P\left\{X \ge \E X + t \right\} \le
    \begin{cases}
        \exp{(-\frac{t^2}{2 \nu^2})}, &~\text{if} \quad 0 \le t \le \tfrac{\nu^2}{b}, \\
        \exp{(-\frac{t}{2 b})}, &~\text{if} \quad t\ge \tfrac{\nu^2}{b}.
    \end{cases}
    \]
\end{lemma}

\end{document}